\documentclass{cedram-alco}
% ...
\usepackage{graphicx}
\usepackage{amssymb,amsmath,amsthm,amsfonts,amscd}
\usepackage{hyperref}
\usepackage{color}
\usepackage{booktabs}
\usepackage{tabularx}
\usepackage{enumitem}
\usepackage[retainorgcmds]{IEEEtrantools}

\usepackage{tikz}
\usetikzlibrary{positioning,automata,arrows,shapes,snakes}
\tikzset{main node/.style={circle,draw,minimum size=0.3em,inner
sep=0.5pt}}
\tikzset{state node/.style={circle,draw,minimum size=2em,fill=blue!20,inner
sep=0pt}}
\tikzset{small node/.style={circle,draw,minimum size=0.5em,inner
sep=2pt,font=\sffamily\bfseries}}
\usetikzlibrary{matrix, arrows, decorations.pathmorphing}

\newcommand{\N}{\mathbb{N}}

\newcommand{\Z}{\mathbb{Z}}

\newcommand{\cala}{\mathcal{A}}
\newcommand{\la}{\mathbf{a}}

\newcommand{\ra}{\rightarrow}
\newcommand{\se}{\subseteq}

\newcommand{\inverse}{^{-1}}

\newcommand{\abs}[1]{\lvert #1 \rvert}

\equalenv{corollary}{coro}
\equalenv{remark}{rema}
\equalenv{example}{exam}
\equalenv{definition}{defi}

\title[Classification of $\la(2)$-finite Coxeter groups]{Classification of Coxeter groups with finitely many elements of
$\la$-value 2}

\author[\initial{R.} \middlename{M.} Green]{\firstname{R.} \middlename{M.} \lastname{Green}}
\address{Department of Mathematics\\
University of Colorado Boulder, Campus Box 395\\
Boulder, Colorado\\
USA, 80309}
\email{rmg@colorado.edu}

\author[\initial{T.} Xu]{\firstname{Tianyuan} \lastname{Xu}}
\address{Department of Mathematics and Statistics\\
 Queen's University\\
Kingston, Ontario\\ Canada, K7L 3N6}
\email{tx7@queensu.ca}


\begin{document}

\begin{abstract}
  We consider Lusztig's $\la$-function on Coxeter groups (in the equal parameter
  case) and classify all Coxeter groups with finitely many elements of
  $\la$-value 2 in terms of Coxeter diagrams.  
\end{abstract}

\maketitle

\section{Introduction}
This paper concerns Lusztig's $\la$-function on Coxeter groups. The
$\la$-function was first defined for finite Weyl groups via their Hecke
algebras by Lusztig in
\cite{Mu}; subsequently, the definition was extended to affine Weyl groups in \cite{Lusztig_II} and to arbitrary Coxeter groups in 
\cite{LG}. The $\la$-function is intimately related to the study of
Kazhdan--Lusztig cells in Coxeter groups, the construction of Lusztig's asymptotic Hecke algebras, and the representation theory of Hecke
algebras; see, for example, \cite{Mu},\cite{Lusztig_II},\cite{LG},
\cite{Geck_a} and \cite{Geck_cellular}.

For any Coxeter group $W$ and $w\in W$, $\la(w)$ is a non-negative integer 
obtained from the structure constants of the Kazhdan--Lusztig basis of the
Hecke algebra of $W$. While $\la$-values are often difficult to compute
directly, it is known that $\la(w)=0$ if and only if $w$ is the identity
element and that $\la(w)=1$ if and only if $w$ is a non-identity element with a
unique
reduced word (see Proposition \ref{subregular}). If we define $W$ to be
\emph{$\la(n)$-finite} for $n\in \Z_{\ge 0}$ if $W$ contains
finitely many elements of $\la$-value $n$ and \emph{$\la(n)$-infinite}
otherwise, then it is also known that $W$ is $\la(1)$-finite if and only if each connected component
of the Coxeter diagram of $W$ is a tree
and contains at most one edge of weight higher than 3 (see Proposition
\ref{tree}). The goal of this paper is to obtain a similar classification of
$\la(2)$-finite Coxeter groups in terms of Coxeter diagrams. 

Our interest in $\la(2)$-finite Coxeter groups comes from considerations about
the asymptotic Hecke algebra $J$ of $W$. This is an associative algebra
which may be viewed as a ``limit'' of the Hecke algebra of $W$, and each two-sided
Kazhdan--Lusztig cell $E\se W$ gives rise to a subalgebra $J_E$ of $J$ (see
\cite{LG}, Section 18).
While $J$ has been interpreted geometrically for Weyl and affine Weyl
groups by Bezrukavnikov et al. in
\cite{Bez1}, \cite{Bez3} and \cite{Bez2}, it is not well understood for
other Coxeter groups, and one approach to understand $J$ in these cases is to
start with the subalgebras $J_E$ in the case $E$ is finite, whence $J_E$ is a
\emph{multi-fusion ring} in the sense of \cite{EGNO}. As the $\la$-function is known to be
constant on each cell, the presence of $\la(2)$-finite groups in our
classification that are not Weyl groups or affine Weyl groups potentially offers
interesting examples of multi-fusion rings of the form $J_E$ where $E$ is a
cell of $\la$-value 2. (For a study of algebras of the form $J_E$ where $E$ is a cell of
$\la$-value 1, see \cite{Xu}.)

We now state our main results. For any (undirected) graph $G$, we define a \emph{cycle} in $G$ to be a sequence
      $C=(v_1,v_2,\cdots, v_n,v_1)$ involving $n$ distinct vertices such that $n\ge 3$ and $\{v_1,v_2\},
      \cdots$, $\{v_{n-1},v_n\}$ and $\{v_n,v_1\}$ are all edges in $G$, and we
      say $G$ is \emph{acyclic} if it contains no
      cycle. Our first main theorem is the following. 
      \begin{thm}
  \label{first theorem}
  Let $W$ be an irreducible Coxeter group with Coxeter diagram $G$.
  \begin{enumerate}
    \item If $G$ contains a cycle, then $W$ is $\la(2)$-finite if and only if
      $G$ is a complete graph.
    \item If $G$ is acyclic, then
      $W$ is $\la(2)$-finite if and only if $G$ is one of the graphs in Figure
      \ref{fig:finite E}, where $n$ denotes the number of vertices in a graph whenever it appears as
  a subscript in the label of the graph and there are $q$ and $r$ vertices
  strictly to the left and the right of the trivalent vertex in $E_{q,r}$.

      \begin{figure}
    \begin{tikzpicture}

      \node(1) {$A_n$};
      \node[main node] (11) [right=0.5cm of 1] {};
      \node[main node] (12) [right=1cm of 11] {};
      \node[main node] (13) [right=1cm of 12] {};
      \node[main node] (15) [right=1.5cm of 13] {};
      \node[main node] (16) [right=1cm of 15] {};
      \node (17) [right=0.5cm of 16] {$(n\ge 1)$}; 

      \path[draw]
      (11)--(12)--(13)
      (15)--(16);

      \path[draw,dashed]
      (13)--(15);

      \node(2) [below=0.5cm of 1] {$B_n$};

      \node[main node] (21) [right=0.5cm of 2] {};
      \node[main node] (22) [right=1cm of 21] {};
      \node[main node] (23) [right=1cm of 22] {};
      \node[main node] (25) [right=1.5cm of 23] {};
      \node[main node] (26) [right=1cm of 25] {};
      \node (27) [right=0.5cm of 26] {$(n\ge 2)$};

      \path[draw]
      (21) edge node [above] {$4$} (22)
      (22)--(23)
      (25)--(26);

      \path[draw,dashed]
      (23)--(25);

      \node (c) [below=1cm of 2] {$\tilde C_n$};
      \node[main node] (c1) [right=0.5cm of c] {};
      \node[main node] (c2) [right=1cm of c1] {};
      \node[main node] (c3) [right=1cm of c2] {};
      \node[main node] (c5) [right=1.5cm of c3] {};
      \node[main node] (c6) [right=1cm of c5] {};
      \node (c7) [right=0.5cm of c6] {$(n\ge 5)$};

      \path[draw]
      (c1) edge node [above] {$4$} (c2)
      (c2)--(c3)
      (c5) edge node [above] {$4$} (c6);

      \path[draw,dashed]
      (c3)--(c5);

      \node (3) [below=1cm of c] {$E_{q,r}$};

      \node[main node] (31) [right=0.4cm of 3] {};
      \node[main node] (32) [right=1cm of 31] {};
      \node[main node] (34) [right=1.5cm of 32] {};
      \node[main node] (38) [right=1.5cm of 34] {};
      \node[main node] (39) [right=1cm of 38] {};
      \node (40) [right=0.5cm of 39] {($q,r\ge 1$)}; 
      \node[main node] (a) [above=0.8cm of 34] {};

      \path[draw]
      (34)--(a)
      (31)--(32)
      (38)--(39);
      \path[draw,dashed]
      (32)--(34)--(38);

      \node(4) [below=0.5cm of 3] {${F}_n$};

      \node[main node] (41) [right=0.5cm of 4] {};
      \node[main node] (42) [right=1cm of 41] {};
      \node[main node] (43) [right=1cm of 42] {};
      \node[main node] (45) [right=1.5cm of 43] {};
      \node[main node] (46) [right=1cm of 45] {};
      \node (47) [right=0.5cm of 46] {$(n\ge 4)$};

      \path[draw]
      (41)--(42)
      (42) edge node [above] {$4$} (43)
      (45)--(46);
      \path[draw,dashed]
      (43)--(45);

      \node(h) [below=0.5cm of 4] {$H_n$};

      \node[main node] (h1) [right=0.45cm of h] {};
      \node[main node] (h2) [right=1cm of h1] {};
      \node[main node] (h3) [right=1cm of h2] {};
      \node[main node] (h5) [right=1.5cm of h3] {};
      \node[main node] (h6) [right=1cm of h5] {};
      \node (h7) [right=0.5cm of h6] {$(n\ge 3)$};

      \path[draw]
      (h1) edge node [above] {$5$} (h2)
      (h2)--(h3)
      (h5)--(h6);
      \path[draw,dashed]
      (h3)--(h5);

      \node(i) [below=0.5cm of h] {$I_2(m)$};

      \node[main node] (i1) [right=0.25cm of i] {};
      \node[main node] (i2) [right=1cm of i1] {};
      \node (i3) [right=0.5cm of i2] {$(5\le m\le \infty)$};

      \path[draw]
      (i1) edge node [above] {$m$} (i2);
    \end{tikzpicture}
    \caption{Irreducible $\la(2)$-finite Coxeter groups with acyclic diagrams.}
    \label{fig:finite E}
  \end{figure}
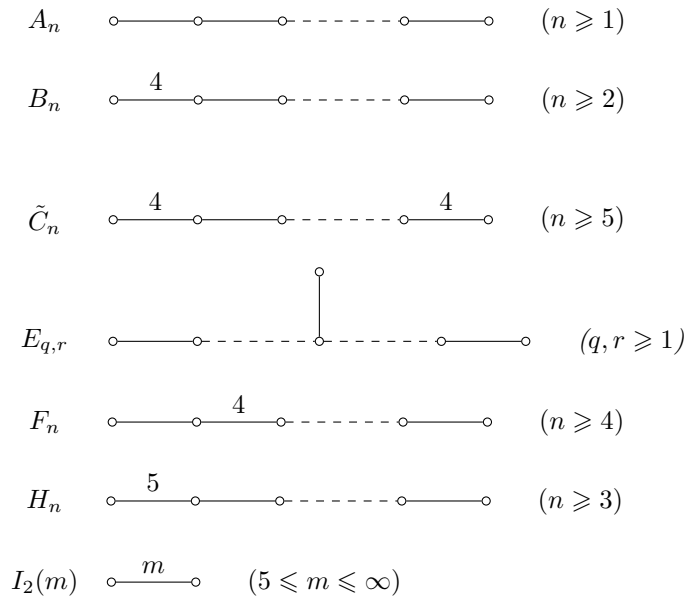

  \end{enumerate}
\end{thm}

\begin{remark}
  \label{DE}
  When $q=1$, $E_{q,r}$ coincides with the Coxeter diagram for the
  Weyl group $D_{r+3}$. When $q=2$ and $r=2,3,4$, $E_{q,r}$ coincides with
  the Coxeter diagram of the Weyl group $E_6, E_7$ and $E_8$, respectively. More generally,
  for any larger value of $r$, $E_{2,r}$ coincides with $E_{r+4}$ in the
  notation of \cite{FC}, which is considered an extension of the type 
  $E$ Coxeter diagrams. In Section \ref{case 2}, we will recall a result from
  \cite{FC} which uses the notations $D_n$ and $E_n$.
\end{remark}
Our second main theorem reduces the classification of reducible
$\la(2)$-finite Coxeter groups to that of irreducible Coxeter groups in the
following sense:

\begin{thm}
  \label{second theorem}
  Let $W$ be a reducible Coxeter group with Coxeter diagram $G$. Let $G_1,G_2,\cdots, G_n$ be the
  connected components of $G$, and let $W_1,W_2,\cdots, W_n$ be their
  corresponding Coxeter groups, respectively.  Then the following are equivalent.
  \begin{enumerate}
    \item $W$ is $\la(2)$-finite.
    \item The number $n$ is finite, i.e. $G$ has finitely many connected components, and
    $W_i$ is both $\la(1)$-finite and
      $\la(2)$-finite for each $1\le i\le n$. 
    \item The number $n$ is finite, and for each $1\le i\le n$, $G_i$ is a graph of the form $A_n (n\ge 1), B_n
      (n\ge 2),
      E_{q,r} (q,r\ge 1), F_n (n\ge 4), H_n(n\ge 3)$ or $I_2(m) (5\le m\le
      \infty)$, i.e.
      $G_i$ is a graph from Figure \ref{fig:finite E} other than $\tilde{C}_n
      (n\ge 5)$.
  \end{enumerate}
\end{thm}

Of the claims in the theorems, Part (2) of Theorem \ref{first theorem}, i.e.
the classification of $\la(2)$-finite Coxeter groups with acyclic Coxeter diagrams, turns
out to require the most amount of work. We describe our strategy for its
proof below. 
A key fact we shall use is that each element of
$\la$-value 2 in a Coxeter group must be \emph{fully commutative} in the sense of Stembridge
(see Section \ref{fc}). This implies, in particular, that we may
associate to any element $w$ with $\la(w)=2$ a poset called its \emph{heap}, a notion
well-defined for any fully commutative element. Heaps of fully commutative
elements will be a fundamental tool for this paper.

In showing that
$W$ is $\la(2)$-finite if $G$ is a graph in Figure \ref{fig:finite E}, the
full commutativity of elements of $\la$-value 2 will reduce our work to the
cases $G=I_2(\infty), G=\tilde C_n$ or
$G=E_{q,r}$ where $\min(q,r)\ge 3$. Indeed, thanks to a result of Stembridge's in
\cite{FC}, $W$ contains finitely many fully
commutative elements if $G$ is any other graph from Figure \ref{fig:finite E},
so $W$ must be $\la(2)$-finite in these cases. It will be easy to show that $W$ is
$\la(2)$-finite when $G=I_2(\infty)$, and the case $G=\tilde C_n$ will also be
easy thanks to a result of Ernst from \cite{Ernst} on the
generalized Temperley--Lieb algebra of type $\tilde C_n$, therefore the only
case requiring more work is $G=E_{q,r}$ where $\min(q,r)\ge 3$. We will prove $W$ is
$\la(2)$-finite in this case via a series of lemmas in Section \ref{case 2},
using arguments that involve heaps.

To show that $G$ must be a graph in Figure \ref{fig:finite E} if $W$ is
$\la(2)$-finite, we first prove that $W$ would be $\la(2)$-infinite whenever $G$ 
contains certain subgraphs, then show that to avoid these subgraphs $G$ has to
be in Figure \ref{fig:finite E}. For each of these subgraphs, we will construct
an infinite family of fully commutative elements that we call ``witnesses'' and verify that they have
$\la$-value 2. We will use three methods for these verifications:

\begin{enumerate}[leftmargin=2em]
    \item First, we
recall a powerful result of Shi from \cite{Shi} that says each fully
commutative element $w$ in a Weyl or affine Weyl group satisfies
$\la(w)=n(w)$, where $n$ is a statistic defined using heaps. We prove
the same result for \emph{star reducible} groups (in the sense of
\cite{star_reducible}, see Proposition \ref{star reducible a}), and use these results to show our witnesses have
$\la$-value 2 by showing they have $n$-value 2. 
\item In our second method, we recall
that the $\la$-function is constant on each two-sided Kazhdan--Lusztig cells of
$W$ and that each cell is closed under the so-called \emph{generalized star
operations} (see Section \ref{star}),
then show our witnesses have $\la$-value 2 by relating them to elements of
$\la$-value 2 by these operations. 
\item 
  In our third and most technical method, we
again show our witnesses have $\la$-value 2 by showing they are in the same
cell as some other element of $\la$-value 2, but the proof will require more
careful arguments involving certain leading coefficients, or
``$\mu$-coefficients'', from Kazhdan--Lusztig polynomials.
\end{enumerate}

The rest of the paper is organized as follows. In Section 2, we briefly recall
the background on 
Coxeter groups and Hecke algebras leading to the definition of the $\la$-function, as well as the
definition and some properties of Kazhdan--Lusztig cells. In Section 3, we introduce our main
technical tools for computing and verifying $\la$-values, namely, generalized
star operations and heaps of fully commutative elements. Sections 4 and 5 prove
the sufficiency and necessity of the diagram criterion of Theorem
\ref{first theorem}, respectively. The first three subsections of Section 5
contain a list of lemmas on the subgraphs that $G$ must avoid for $W$ to be
$\la(2)$-finite. The lemmas are grouped according to the method of verifying
the witnesses' $\la$-values, with \ref{mu arguments} being the most technical
part of the paper. We prove Theorem \ref{second theorem} in Section 6.
Finally, in Section \ref{conclusion} we briefly discuss several open problems naturally arising from
this paper, as well as their connections to other works in the literature. The
problems include generalizing the aforementioned result of Shi to arbitrary
Coxeter groups, enumerating 
elements of $\la$-value 2 in $\la(2)$-finite Coxeter groups, and the
classification of $\la(3)$-finite Coxeter groups.

\section{Preliminaries}
The $\la$-function arises from the Kazhdan--Lusztig theory of Coxeter groups. We
review the relevant basic notions and facts in this section. Besides defining $\la$, we will recall the definitions of Kazhdan--Lusztig cells and the
``$\mu$-coefficients'' of Kazhdan--Lusztig polynomials. Both these notions will be key
to the proofs of our main theorems.

\subsection{Coxeter groups} 
\label{Coxeter def}
Throughout the article, $W$ shall denote a Coxeter
group with a finite generating set $S$ and Coxeter matrix $M=[m(s,t)]_{s,t\in S}$. Thus,
$m(s,s)=1$ for all $s\in S$, $m(s,t)=m(t,s)\in \Z_{\ge 2}\cup \{\infty\}$ for all
distinct $s,t\in S$, and $W$ is generated by $S$ subject to the relations
$(st)^{m(s,t)}=1$ for all $s,t$ for which $m(s,t)$ is finite. 

The defining data of each Coxeter group can be encoded via its \emph{Coxeter diagram}.
This is the weighted, undirected graph with vertex set $S$ and edge set
$\{\{s,t\}:s,t\in S,m(s,t)\ge 3\}$ such that each edge $\{s,t\}$ has weight
$m(s,t)$.  Each edge is labelled by its weight except when the weight is $3$. 
A Coxeter group is called \emph{irreducible} if its Coxeter diagram is
connected; otherwise the group is \emph{reducible}. Note that any reducible
Coxeter group $W$ with Coxeter diagram $G$ is isomorphic to the direct product
of the Coxeter groups encoded by the connected components of $G$.

Let $S^*$ be the free monoid generated by $S$. For any $w\in W$, we define the
\emph{length} of $w$, written $l(w)$, to be the minimum length of all words in
$S^*$ that express $w$. We call any such minimum-length word a \emph{reduced word} of
$w$. For any distinct $s,t\in S$, we call the relation
\[
  sts\cdots=tst\cdots 
\] 
where both sides have $m(s,t)$ factors a \emph{braid relation}. Since
$s^2=(ss)^{m(s,s)}=1$ for all $s\in S$, the
braid relation is equivalent to
the relation $(st)^{m(s,t)}=1$ from the definition of $W$. When $m(s,t)=2$, we call the relation $st=ts$ a \emph{commutation
relation}, for $s$ and $t$ commute. 

We can now recall the useful Matsumoto--Tits Theorem.
\begin{prop}[\cite{Tits};{\cite[Theorem 1.9]{LG}}]
  \label{Matsumoto-Tits}
  Let $w\in W$. Then any pair of reduced words of $w$ can be obtained from each
  other by a finite sequence of braid relations.
\end{prop}

For more basic notions and facts about Coxeter groups such as 
the Bruhat order and its subword property, see \cite{BB}.

\subsection{The $\la$-function}
Let $W$ be an arbitrary Coxeter group. We recall Lusztig's definition of the function
$\la: W\ra \Z_{\ge 0}$ below.

Let $\cala=\Z[v,v\inverse]$. Following \cite{LG}, we define the \emph{Hecke
algebra} of $W$ to be the unital $\cala$-algebra $H$ generated by the set
$\{T_s:s\in S\}$ subject to the relations 
\begin{equation}
  (T_s-v)(T_s+v\inverse)=0
  \label{eq:our normalization}
\end{equation}
for all $s\in S$ and 
\[
  T_sT_tT_s\cdots = T_tT_sT_t\cdots
\]
for all $s,t\in S$, where both sides have $m(s,t)$ factors.

It is well-known
that $H$ has a \emph{standard basis} $\{T_w:w\in W\}$ where
$T_w=T_{s_1}\cdots T_{s_q}$ for any reduced word $s_1\cdots s_q$ ($s_1,\cdots,
s_q\in S$) of $w$, as well as a \emph{Kazhdan--Lusztig basis} $\{C_w:w\in W\}$
with remarkable properties (see \cite{EW}). Now let $h_{x,y,z} (x,y,z\in W)$ be the elements of $\cala$
such that
\[
  C_xC_y=\sum_{z\in W} h_{x,y,z}C_z
\]
for all $x,y$. By Lemma 13.5 of \cite{LG}, for each $z\in W$, there exists a
unique integer $\la(z)\ge 0$ that satisfies the conditions
\begin{enumerate}
  \item $h_{x,y,z}\in v^{\la(z)}\Z[v\inverse]$ for all $x,y\in W$,
  \item $h_{x,y,z}\not\in v^{\la(z)-1}\Z[v\inverse]$ for some $x,y\in W$.
\end{enumerate}
This defines the function $\la: W\ra\Z_{\ge 0}$.

Elements of $\la$-value 0 or 1 are well understood in the following sense. 
\begin{prop}[{\cite[Proposition 13.7]{LG}}, {\cite[Corollary 4.10]{Xu}}]
  \label{subregular}
  Let $W$ be an arbitrary Coxeter group, and let $1_W$ be the identity of
  $W$. For all $w\in W$, we have
  \begin{enumerate}
    \item $\la(w)=0$ if and only if
      $w=1_W$.
    \item $\la(w)=1$ if and only if $w\neq 1_W$ and $w$ has a
      unique reduced word.
  \end{enumerate}
\end{prop}
\noindent Here, the set of non-identity elements with a unique reduced word is
known to be a \emph{two-sided Kazhdan--Lusztig
cell} (which we will define in the next subsection), and is sometimes called the \emph{subregular cell} (see \cite{subregular}
and \cite{Xu}). 

The following result classifies $\la(1)$-finite Coxeter groups, i.e. Coxeter
groups with finite subregular cells, in
terms of Coxeter diagrams. We will use it in the proof of Theorem
\ref{second theorem} in Section \ref{reducible}.
\begin{prop}[{\cite[Proposition 3.8]{subregular}}]
  \label{tree}
  Let $W$ be an irreducible Coxeter group with Coxeter diagram $G$. Then $W$ is
  $\la(1)$-finite if and
  only if $G$ is a tree and there is at most one edge of weight higher than 3
  in $G$.
\end{prop}
\noindent 

Besides the identity element and the elements in the subregular cell, it is
also easy to compute the $\la$-values
of products of commuting generators in a 
Coxeter group, thanks to the following two results of Lusztig.
\begin{prop}[{\cite[Section 14]{LG}}]
  \label{a local}
  Let $W$ be a Coxeter group with generating set $S$. Let $I\se S$ and let $W_I$ be the subgroup of $W$ generated by
  $I$. If $w\in W_I$, then $\la(w)$ computed in terms of $W_I$ is equal to
  $\la(w)$ computed in terms of $W$.
\end{prop}

\begin{remark}
  The above statement appears as part of Conjecture 14.2 in \cite{LG}. However, it is
  known to hold in the setting of this paper, which is called the
  \emph{equal parameter} or the \emph{split} case (see \cite{LG}, Section 15).
  The same remark applies to Proposition \ref{constant}.
\end{remark}

\begin{prop}[{\cite[Proposition 13.8]{LG}}]
  \label{a for longest}
  Let $W$ be a finite Coxeter group, and let
  $w_0$ be the longest element of $W$. Then $\la(w_0)=l(w_0)$.
\end{prop}

\begin{corollary}
  \label{product a}
  Let $W$ be a Coxeter group with generating set $S$. Let $I=\{s_1,s_2,\cdots$,
  $s_k\}$
  be a subset of $S$ such that $m(s_i,s_j)=2$ for all $1\le i<j\le k$, and let
  $w_0=s_1s_2\cdots s_k$. Then $\la(w_0)=k$.
\end{corollary}
\begin{proof}
  The elements of $I$ commute with each other since $m(s_i,s_j)=2$
  for all distinct $i,j$, therefore the subgroup $W_I$ of $W$ generated by
  $I$ is isomorphic to the direct product of $k$ copies of the cyclic group of
  order 2. In particular, $W_I$ is finite. Furthermore, $w_0$ is clearly the
  longest element of $W_I$, therefore $\la(w_0)=l(w_0)=k$ by propositions
  \ref{a local} and \ref{a for longest}.
\end{proof}

\subsection{Kazhdan--Lusztig cells}
We define the Kazhdan--Lusztig cells of a Coxeter group $W$ in this subsection. 
Let $H$ be the Hecke algebra of $W$, and let $\{C_w:w\in W\}$
be the Kazhdan--Lusztig basis of $H$. For each $x\in W$, let $D_x:H\ra \cala$
be the linear map such that 
\[
  D_{x}(C_y)=\delta_{x,y}
\]
for all $y\in W$, where $\delta$ is the Kronecker delta symbol. 
Furthermore, for $x,y\in W$,
\begin{enumerate}[leftmargin=2em]
  \item define $x\prec_L y$ if  $D_{x}(C_sC_y)\neq 0$ for some
    $s\in S$;
  \item define $x\le_L y$ if there is a sequence $x=z_1,z_2,\cdots,z_n=y$ in
    $W$ such that $z_{i}\prec_L z_{i+1}$ for all $1\le i\le n-1$;
  \item define $x\sim_L y$ if $x\le_L y$ and $y\le_L x$.
\end{enumerate}
By the construction, $\sim_L$ defines an equivalence
relation on $W$. 
We call the equivalence classes the
\emph{left Kazhdan--Lusztig cells}, or simply the \emph{left cells}, of $W$,
and we define the  \emph{right (Kazhdan--Lusztig) cells}
and $\emph{two-sided (Kazhdan--Lusztig) cells}$ of $W$ similarly. Here, to
define the two-sided cells, start by declaring $x\prec_{LR} y$ if either
$D_{x}(C_sC_y)\neq 0$ for some $s\in S$ or $D_x(C_yC_s)\neq 0$ for some $s$
(i.e. if either $x\prec_L y$ or $x\prec_R y$).
Note that each two-sided cell of $W$ must be a union of left cells as well as a
union of right cells. 

The Kazhdan--Lusztig cells of $W$ have the following key connection with the
$\la$-function on $W$.

\begin{prop}[{\cite[Section 14]{LG}}]
  \label{constant}
  Let $x,y\in W$. If $x\le_{LR} y$, then $\la(x)\ge \la(y)$. In particular, if
  $x\sim_{LR} y$, then $\la(x)=\la(y)$.
\end{prop}

By the proposition, one way to establish that an element $x\in W$ has a certain
$\la$-value is to find another element $y$ of that $\la$-value and prove that
$x$ and $y$ are in the same cell. We will repeatedly use this strategy in
sections \ref{star arguments} and \ref{mu arguments}.

As we may see from their construction, the key to understanding Kazhdan--Lusztig cells lies in
understanding the products of the form $C_sC_y$. These products are controlled
by the \emph{Kazhdan--Lusztig polynomials}, which are defined to be the elements $p_{x,y}\in \cala\, (x,y\in W)$
such that 
\[
  C_y=\sum_{x\in W} p_{x,y}T_x
\]
for all $y\in W$, where the elements $\{T_w:w\in W\}$ form the standard basis
of $T$. More precisely, for each $x,y\in W$, let $\mu_{x,y}$ be the coefficient of the term
$v^{-1}$ in $p_{x,y}$, then we have the following formulae.
\begin{prop}[{\cite[Theorem 6.6]{LG}}]
  \label{KL basis mult}
  Let $y\in W$, $s\in S$, and let $\le$ be the Bruhat order on $W$. Then in the
  Hecke algebra $H$ of $W$,
  \begin{eqnarray*}
    C_s C_y&=& 
    \begin{cases}
      (v+v\inverse) C_y&\quad\qquad\text{if}\quad sy<y,\\
      C_{sy}+\sum\limits_{x:sx<x<y} \mu_{x,y}C_x&\quad\qquad \text{if} \quad
      sy>y,
    \end{cases}\\
    &&\\
    C_y C_s&=& 
    \begin{cases}
      (v+v\inverse) C_y&\quad \text{if}\quad ys<y,\\
      C_{ys}+\sum\limits_{x:xs<x<y} \mu_{x^{-1},y^{-1}}C_x&\quad \text{if}
      \quad ys>y.
    \end{cases}
    \\
  \end{eqnarray*}
\end{prop}

\begin{remark}
  \label{symmetry}
  It is known that $\mu_{x,y}=\mu_{x\inverse,y\inverse}$ for any $x,y\in W$
  (see Section 5.6 and Corollary 6.5 of \cite{LG}),
  therefore the last formula in the proposition also holds with $\mu_{x,y}$ in
  place of $\mu_{x\inverse,y\inverse}$.
\end{remark}

\begin{remark}
  \label{conversion}
  The paper \cite{KL} uses a normalization of the Hecke algebra that is
  different from ours, namely, it uses the relation $(T_s+1)(T_s-q)=0$ in place
  of our Equation
  \eqref{eq:our normalization}. Consequently, the
  Kazhdan--Lusztig polynomials $P_{x,y}$ obtained in \cite{KL}---which are
  polynomials in $q$---do not exactly
  agree with our Kazhdan--Lusztig polynomial $p_{x,y}$. However, it is
  straightforward to 
  check that we may convert $P_{x,y}$ to $p_{x,y}$ by first substituting
  $q$ by $v^2$ in $P_{x,y}$ and then multiplying the result by $v^{l(x)-l(y)}$. 
  In particular, our definition of the numbers $\mu_{x,y}$ agrees with that in
  \cite{KL}.
\end{remark}

\noindent The $\mu$-coefficients are often called the ``leading coefficients
of Kazhdan--Lusztig polynomials'' in the literature. Note that the elements $C_s (s\in S)$ generate $H$ by Proposition
\ref{KL basis mult}, so in a sense the $\mu$-coefficients control the
multiplication of the Kazhdan--Lusztig basis elements in the Hecke algebra. As
such, they also lead to an alternative characterization of the relations
$\prec_L$ and $\prec_R$: for each $y\in W$, define the \emph{left descent set} and
\emph{right descent set} of $y$ to be the sets
\[
  \mathcal{L}(y)=\{s\in S:sy<y\},\qquad \mathcal{R}(y)=\{s\in S:ys<y\},
\]
respectively. Then the following proposition holds.

\begin{prop}
  \label{alternative prec}
  Let $x,y\in W$. Then
  \begin{enumerate}
    \item $x\prec_L y$ if and only if one of the following conditions holds:
      \textup{(}a\textup{)} $x=y\neq 1_W$;
      \textup{(}b\textup{)} $x=sy$ for some $s\notin
      \mathcal{L}(y)$; \textup{(}c\textup{)} $x<y$, $\mathcal{L}(x)\not\se \mathcal{L}(y)$, and
      $\mu_{x,y}\neq 0$.
    \item $x\prec_R y$ if and only if one of the following conditions holds:
      \textup{(}a\textup{)} $x=y\neq 1_W$;
      \textup{(}b\textup{)} $x=ys$ for some $s\notin
      \mathcal{R}(y)$; \textup{(}c\textup{)} $x<y$, $\mathcal{R}(x)\not\se \mathcal{R}(y)$, and
      $\mu_{x,y}\neq 0$.
  \end{enumerate}
\end{prop}
\begin{proof}
  By Proposition \ref{KL basis mult}, we have  $D_x(C_sC_y)\neq 0$ for some $s\in S$ if and only if one of the following occurs:
  \begin{enumerate}
    \item[(a)] $x=y\neq 1_W$, so that $\mathcal{L}(y)\neq \emptyset$ and $D_x(C_sC_y)\neq 0$ for each $s\in
      \mathcal{L}(y)$;
    \item[(b)] $C_x$ appears in $C_sC_y$ for some $s\not\in \mathcal{L}(y)$, with $x=sy$;
    \item[(c)] $C_x$ appears in $C_sC_y$ for some $s\not\in \mathcal{L}(y)$, and $x$
      satisfies $x<y, sx<x$ and $\mu_{x,y}\neq 0$. Note that in this case
        we have    $\mathcal{L}(x)\not\se\mathcal{L}(y)$ as $s\in
            \mathcal{L}(x)\setminus \mathcal{L}(y)$. Conversely, 
      if $x<y$, $\mathcal{L}(x)\not\se\mathcal{L}(y)$, and $\mu_{x,y}\neq 0$,
      then 
      $D_x(D_sD_y)\neq 0$ for each $s\in \mathcal{L}(x)\setminus
      \mathcal{L}(y)$ by Proposition \ref{KL basis mult}.
  \end{enumerate}
  Statement (1) now follows. The proof of (2)
  is similar.
\end{proof}

In propositions \ref{star and mu} and \ref{star relation}, we will describe
ways to compute certain
$\mu$-coefficients combinatorially without referring to the Hecke algebra.
This will allow us to avoid difficult computations of Kazhdan--Lusztig
polynomials and understand Kazhdan--Lusztig cells by using only the combinatorics of
Coxeter groups.

To end this section, we record several facts for future use.

\begin{corollary} \label{weak order} Let $x,y\in W$, and let $s\in
S$.  \begin{enumerate} \item We have $sy\le_L y$ if $sy>y$, and $ys\le_R y$ if $ys>y$;
    \item If there exist elements $u,v\in W$ such that $y=uxv$ and
      $l(y)=l(u)+l(x)+l(v)$ where $l$ is the length function on
      $W$ \textup{(}see Section \ref{Coxeter def} for the definition of
      $l$\textup{)}, then
      we have $y\le_{LR}x$ and $\la(y)\ge \la(x)$.  \end{enumerate} \end{corollary}
      \begin{proof} This is a simple corollary of propositions \ref{KL basis
        mult} and \ref{constant}. Note that (2) follows from repeated application of (1) and
        Proposition \ref{constant}, hence it suffices to prove (1). Suppose
        $sy>y$. Then $D_{sy}(C_sC_y)=1$  by Proposition \ref{KL basis mult},
        therefore $sy\prec y$ and $sy\le_L y$ by definition. Similarly,
        we have $ys\le_R y$ if $ys>y$.  \end{proof}

      \begin{prop}[{\cite[Proposition 5.4]{LG}}]
        \label{degree bound}
        Let $x,y\in W$. If $x\le y$, then $p_{x,y}=v^{-l(y)+l(x)}\mod
        v^{-l(y)+l(x)+1}\Z[v]$.
      \end{prop}

      \begin{corollary}
        \label{difference 1}
        Let $x,y\in W$. If $x\le y$ and $l(x)=l(y)-1$, then $p_{x,y}=v\inverse$
        and hence $\mu_{x,y}=1$.
      \end{corollary}
      \begin{proof}
        This is immediate from the
        well-known fact that $p_{x,y}\in \Z[v\inverse]$ (see Section 5.3 of
        \cite{LG}) and  Proposition \ref{degree bound}.
      \end{proof}

      \begin{prop}[{\cite[Fact 5]{Warrington}}]
        \label{extremal}
        Let $x,y\in W$ be such that $l(x)<l(y)-1$. If $\mathcal{L}(y)\not\se
        \mathcal{L}(x)$ or $\mathcal{R}(y)\not\se\mathcal{R}(x)$, then $\mu(x,y)=0$.
      \end{prop}


      \section{Tools for computation of $\la$}
      \label{sec:tools}
      We introduce our main tools for verification and computation of
      $\la$-values in this section. The first tool is the so-called
      generalized star operations, which we will often use to show two elements
      are in a same Kazhdan--Lusztig cell and hence of the same $\la$-value.
      The second tool involves heaps of fully commutative elements and will
      allow us to directly compute $\la$-values in certain cases.

      \subsection{Generalized star operations}
      \label{star}
      We review the notion of a
      \emph{generalized star operation} in this subsection. We highlight a
      direct connection between the operation and Kazhdan--Lusztig cells, then describe
      a more subtle recurrence relation involving the operation and the
      $\mu$-coefficients from Proposition \ref{KL basis mult}. 

      Let $W$ be an arbitrary Coxeter group, and let $s,t\in S$ be a pair of
      generators of $W$ with $3\le m(s,t)<\infty$. Set $I=\{s,t\}$, let $W_I=\langle
      s,t\rangle$, the subgroup of $W$ generated by $s$ and $t$, and set ${}^I W=\{w\in
        W: \mathcal{L}(w)\cap I=\emptyset\}$. It is known that every $w\in W$ admits a unique
        factorization $w=w_I\cdot {}^I w$, called a \emph{coset decomposition}, where
        ${}^I w\in {}^I W$ and $w_I\in W_I$; moreover, we have $l(w)=l(w_I) +
        l({}^Iw)$ in this case. (For proofs
        of these facts and an algorithm to compute the factors $w_I$ and
        ${}^I w$, see Proposition 2.4.4 of \cite{BB}.) Consider the following situations:
        \begin{enumerate}
          \item $w_I=1$;
          \item $w_I$ is the longest element $sts\cdots$ of length $m(s,t)$ in
            $W_I$;
          \item $w$ is one of the $(m-1)$ elements $s\cdot {}^Iw, ts\cdot {}^Iw,
            sts\cdot {}^I
            w,tsts\cdot {}^I w,\cdots$;
          \item $w$ is one of the $(m-1)$ elements $t\cdot {}^Iw, st\cdot {}^Iw,
            tst\cdot {}^I 
            w,stst\cdot {}^I w,\cdots$.
        \end{enumerate}
        We call the sequences appearing in (3) and (4) \emph{left
          $\{s,t\}$-strings} or $\emph{left $I$-strings}$, or simply
     \emph{left strings} if the pair $\{s,t\}$ is clear from context. For any element $w$ in a left
        $\{s,t\}$-string other than the longest, we define ${}^* w$ to be the
        element to the right of $w$. Otherwise, we leave ${}^* w$ undefined. We call
        the map $w\mapsto {}^* w$ the \emph{upper left star operation with respect
        to $I$}. 

        Similarly, we define the \emph{lower left
        star operation} to be the operation $w\mapsto {}_* w$ where $w$ is an element
        in a left string other than the shortest and ${}_* w$ is the element to the
        left of $w$ in the same string. In addition, we say $w$ is \emph{left star reducible to ${}_*
      w$ with respect to $I$} whenever the latter is defined. More generally,
      dropping the reference to a particular pair of generators,  we
      say $y$ is \emph{left star reducible} to $x$ for $x,y\in W$ if there is a sequence
      $x=z_1,z_2,\cdots,z_n=y$ in $W$ such that for each $1\le i\le n-1$, there is
      some pair $I_i=\{s_i,t_i\}\se S$ with $3\le m(s_i,t_i)<\infty$ such that
      $z_{i+1}$ is left star reducible to $z_i$ with respect to $I_i$.

      The concepts and notations above have obvious right-handed counterparts,
      where the coset decompositions to be considered are of the form
      $w=w^I\cdot w_I$ where $w_I\in W_I$ and the factor $w^I$ is from the set $W^I:=\{w\in
        W:\mathcal{R}(w)\cap I=\emptyset\}$. We
      refer to the two types of left star operations and their right-handed
      counterparts collectively as
      \emph{generalized star operations}. Finally, for $x,y\in W$, we say that $y$
      is \emph{star reducible} to $x$ if there is a sequence
      $x=z_1,z_2,\cdots,z_n=y$ in $W$ such that $z_{i+1}$ is either left reducible
      or right reducible to $z_i$ for each $1\le i\le n-1$.

      \begin{remark}
        \label{original star}
        For each pair $I=\{s,t\}\se S$ with $m(s,t)=3$ and each member of a left $\{s,t\}$-string, only one of the lower and upper left star
        operations is defined for
        each member of a left $\{s,t\}$-string. The one that does is simply called
        the \emph{left star operation} in the paper \cite{KL} where the operation
        was first introduced by Kazhdan and Lusztig. Similarly, it makes sense to
        simply speak of a \emph{right star operation} with respect to $I$.  
      \end{remark}

      \begin{example}
        \label{star example}
        Let $W$ be a Coxeter group with generating set $S=\{a,b,c\}$. Suppose
        $m(a,b)=3$, $m(b,c)=4$, $m(a,c)=2$, let $I=\{a,b\}, J=\{b,c\}$, and let $x=abcab$. Then with respect to $I$, the coset decompositions of
        $x$ are given by
        \[x= x_I\cdot {}^I x= aba\cdot cb, \quad x= x^I \cdot x_I = abc\cdot ab.\]
        It follows that $x$ is not in a left $I$-string but is in a right $I$-string.
        Moreover, as pointed out in Remark \ref{original star}, only one right star operation with respect
        to $I$ is defined on $x$ since $m(a,b)=3$: only the lower star
        operation is defined, and we have $x_*=abca$. With respect to
        $J$, the coset decompositions of $x$ are given by 
        \[
        x=x_J\cdot {}^J x=b\cdot abcb, \quad x=x^J\cdot x_J=ba\cdot bcb.\]
        It follows that $x$ is both in a left $J$-string and in a right $J$-string.
        Moreover, we have ${}^{*}x=cbabcb$ and $x_*=babc$, but ${}_*x$ and
        $x^*$ are not defined.
      \end{example}

      Generalized star operations are intimately related to Kazhdan--Lusztig
      cells:

      \begin{prop}     
        \label{star and cell} 
        Let $W$ be an arbitrary Coxeter group, and let  $I=\{s,t\}$ be a pair of
        generators of $W$ for which $3\le m(s,t)<\infty$. Then the following hold,
        where all star operations are performed with respect to $I$.
        \begin{enumerate} 
          \item Let $y$ be an element of a left $\{s,t\}$-string such that 
            ${}_*y$ makes sense, then $y\sim_{L}
            {}_* y$.  
          \item Let $y$ be an element of a right $\{s,t\}$-string such that 
            $y_*$ makes sense, then $y\sim_{R} y_*$.  
        \end{enumerate} 
      \end{prop}

      \begin{remark}
        The above facts are well-known to experts, but we have
        not found a reference stating it explicitly in this way, so we include a
        brief proof below. 
      \end{remark}
      \begin{proof}
        We first prove (1). Without loss of generality, suppose $I\cap
        \mathcal{L}(y)=\{s\}$. Then the definition of left strings guarantees that
        $I\cap\mathcal{L}({}_*y)=\{t\}$. Since ${}_*y<y$, $t\in \mathcal{L}({}_*
        y)\setminus\mathcal{L}(y)$ and $\mu_{{}_*y,y}=1$ by Corollary
        \ref{difference 1}, we have ${}_* y\le_L y$ by Proposition \ref{alternative
        prec}. On the other hand, $y\le_L {}_* y$ by Corollary \ref{weak order},
        therefore $y\sim_L {}_* y$. The proof of (2) is similar.
      \end{proof} 

      \begin{corollary}
        \label{star a}
        Let $x,y\in W$. If $y$ is star reducible to $x$, then $\la(x)=\la(y)$.
      \end{corollary}
      \begin{proof}
        Suppose $y$ is star reducible to $x$. Then $x\sim_{LR} y$ by repeated
        application of Proposition \ref{star and cell}, therefore
        $\la(x)=\la(y)$ by Proposition \ref{constant}.
      \end{proof}

      Generalized star operations are also connected with $\mu$-coefficients:

      \begin{prop}[{\cite[Theorem 4.2]{KL}}]
        \label{star and mu}
        Let $W$ be an arbitrary Coxeter group, and let $I=\{s,t\}$ be a pair of
        generators of $W$ for which $m(s,t)=3$. Then the following hold, where all
        star operations are performed with respect to $I$.
        \begin{enumerate}
          \item  Let $x,y\in W$ be elements of left $\{s,t\}$-strings such that
            $xy^{-1}\notin W_I$. Then $\mu(x,y)=\mu(*x,*y)$, where $*\alpha$ stands for the
            result of applying the left star operation on $\alpha$ for each string
            $\alpha$ \textup{(}see Remark \ref{original star}\textup{)};
          \item  Let $x,y\in W$ be elements of right $\{s,t\}$-strings such that
            $x^{-1}y\notin W_I$. Then $\mu(x,y)=\mu(x*,y*)$, where $\alpha*$ stands for the
            result of applying the right star operation on $\alpha$ for each string
            $\alpha$ \textup{(}see Remark \ref{original star}\textup{)}.
        \end{enumerate}

      \end{prop}

      \begin{prop}[{\cite[Section 10.4]{Mu}}; {\cite[Proposition
        5.9]{Green_Jones}}]
        \label{star relation}
        Let $W$ be an arbitrary Coxeter group, and let  $I=\{s,t\}$ be a pair of
        generators of $W$ for which $3\le m(s,t)<\infty$. Then the following
        hold,
        where all star operations are performed with respect to $I$.

        \begin{enumerate}
          \item 
            Let $x,y\in W$ be
            elements of left $\{s,t\}$-strings such that $\mathcal{L}(x)\cap I\neq
            \mathcal{L}(y)\cap I$. Then
            \[
              \mu({}_* x,y)+\mu({}^* x,y)=\mu(x,{}_* y)+\mu(x, {}^* y);
            \]

          \item 
            Let $x,y\in W$ be
            elements of right $\{s,t\}$-strings such that $\mathcal{R}(x)\cap I\neq
            \mathcal{R}(y)\cap I$. Then
            \[
              \mu(x_*,y)+\mu(x^*,y)=\mu(x,y_*)+\mu(x, y^*).
            \]
        \end{enumerate}
        Here, we define $\mu(\alpha,\beta)=0$ if either $\alpha$ or $\beta$ is an undefined symbol.
      \end{prop}

      In Section \ref{star arguments}, we will frequently use the two propositions above to
      compute certain $\mu$-coefficients $\mu_{x,y}$ recursively, then use Proposition
      \ref{alternative prec} to conclude that $x\sim_{L}y$ or $x\sim_R y$.  This provides
      a very useful connection, albeit a less direct one than Proposition \ref{star
      and cell}, between generalized star operations and cells. 

      \subsection{Full commutativity and heaps} 
      \label{fc}
      Let $W$ be an arbitrary Coxeter group. In this subsection,
      we show that any element with $\la$-value 2 must be
      \emph{fully commutative} in the sense of \cite{FC}. We then recall a
      combinatorial characterization of the $\la$-values of fully commutative
      elements in a Weyl or affine Weyl group in terms of \emph{heaps}. This
      characterization will allow us to compute certain $\la$-values without
      recourse to Kazhdan--Lusztig theory in Section \ref{first lemmas}.

      An element $w\in W$ is said to be \emph{fully commutative} if any pair of
      reduced words of $w$ can be obtained from each other by means of only
      commutation relations.  It is well-known that $w$ is fully commutative if
      and only if no reduced word of $w$ contains a
      contiguous subword of $sts\cdots$ of length $m(s,t)$ where $s,t\in S$
      and $m(s,t)\ge 3$ (see \cite{FC}, Proposition 2.1).

      \begin{remark}
        \label{removal}
        Let $w$ be a fully commutative element with a reduced word $w=stw'$ where
        $l(w)=l(w')+2$ and $m(s,t)\ge 3$. Consider the coset decomposition
        $w=w_I\cdot {}^I w$ with respect to  the pair $I=\{s,t\}$. Since $w$ is
        fully commutative, $w_I$ cannot be the word $sts\cdots$ of length
        $m(s,t)$, therefore $w$ is an element of a left $\{s,t\}$-string, with ${}_* w=tw'$ with respect to $I$. That is, whenever a reduced word
        of a fully commutative element starts with a pair of letters $s,t\in S$
        with $m(s,t)\ge 3$, the lower left star operation with respect to
        $\{s,t\}$ simply removes the leftmost letter of $w$. Similarly,
        whenever a reduced word of a fully commutative element ends with a pair
        of letters $s,t\in S$ with $m(s,t)\ge 3$, the lower right star
        operation with respect to $\{s,t\}$ simply removes the rightmost letter
        of $w$.  \end{remark}

      Problems related to fully commutative elements, such as the classification of
      Coxeter groups with finitely many fully commutative elements, the
      enumeration of fully commutative elements for those groups, and 
      connections fully commutative elements have to Kazhdan--Lusztig cells and
      so-called \emph{generalized
      Temperley--Lieb algebras}, have been studied
      extensively; see, for example,  \cite{BJN},\cite{Fan}, \cite{FCcells}, 
      \cite{Shi}, \cite{FC} and \cite{FC2}.
      Fully commutative elements also provide a suitable framework for studying elements of $\la$-value 2 because of the following fact.

      \begin{prop}
        \label{a2 is fc}
        Let $w\in W$. If $\la(w)=2$, then $w$ is fully commutative. 
      \end{prop}
      \begin{proof}
        We prove the contrapositive of the statement, i.e. that if $w$ is not fully
        commutative, then $\la(w)\neq 2$. 

        Suppose $w$ is not fully commutative. Then $w$ can be written in the form
        $w=uxv$ where $l(w)=l(u)+l(x)+l(v)$ and $x$ is of the form
        $x=sts\cdots$ with $s,t\in S, m(s,t)\ge 3$ and $l(x)=m(s,t)$. By propositions
        \ref{a local} and $\ref{a for longest}$, we have $\la(x)=m(s,t)$ in this case, therefore $\la(w)\ge \la(x)=m(s,t)\ge 3$ by
        Corollary \ref{weak order}.  This completes the proof.
      \end{proof}

      Next, we define the \emph{heap} of an arbitrary word $s_1s_2\cdots s_q$ in the free
      monoid $S^*$: this is the poset $([q],\preccurlyeq)$ where
      $[q]=\{1,2,\cdots,q\}$ and $\preccurlyeq$ is the partial order on
      $[q]=\{1,2,\cdots,q\}$ obtained via the reflexive transitive closure of the relations
      \[
        i\prec j \quad\text{if}\quad i<j\quad\text{and}\quad m(s_i,s_j)\neq 2.
      \]
      In particular, $i\prec j$ if $i<j$ and $s_i=s_j$. We refer to the
      generator $s_i$ as the \emph{label} of $i$ for each $i\in [q]$. It is well-known that the heaps of
      any two words in $S^*$ related by a commutation relation from $W$ are
      isomorphic as posets (see \cite{FC}, Section 2.2), therefore for any fully commutative element $w\in W$,
      it makes sense to define the \emph{heap of $w$} to be the heap of any reduced
      word of $w$.  In this case, we denote the heap of $w$ by $H(w)$. On the other
      hand, given any word in $S^*$, there is also a criterion for
      determining if its heap is that of a fully commutative element:

      \begin{prop}[{\cite[Proposition 3.3]{FC}}]
        \label{fc criterion}
        The heap $P$ of a word $s_1s_2\cdots s_q$ in $S^*$ is the heap of some fully commutative
        element in $W$ if and only if 
        \begin{enumerate}
          \item There is no covering relation $i\prec j$ such that $s_i=s_j$;
          \item There is no convex chain $i_1<i_2\cdots <i_m$ in $P$ such that
            $s_{i_1}=s_{i_3}=\cdots = s$ and $s_{i_2}=s_{i_4}=\cdots=\cdots=t$,
            where $s,t\in S$ and $m = m(s,t)\ge 3$.
        \end{enumerate}
      \end{prop}

      There is an intuitive way to visualize heaps of words in $S^*$. Consider the lattice $S\times \N$, with $S$
      indexing the \emph{columns} of the lattice and $\N$ indexing the \emph{levels}
      or
      \emph{heights}. We say two
      columns $s,t$ are \emph{adjacent} if the corresponding vertices are adjacent in the
      Coxeter graph, i.e. if  $m(s,t)\ge 3$.

      For any word $s_1s_2\cdots s_q\in S^*$, we may
      embed its heap $P$ as a set of lattice points in $S\times \N$ as follows:
      read the word from left to right, and drop a point in the column representing
      $s_i$ as we read each letter. Here, we envision each point as being under the
      influence of ``gravity'' in the sense that the point must fall to the lowest
      possible row in its column subject to one condition, namely, it must fall
      higher than every point that was placed before it in the same column
      or in an adjacent column. We define the index of this lowest possible row
      to be the \emph{level} of the point. 
      
        \begin{remark}
        A poset $P$ is said to be \emph{ranked} if there exists a function
        $\rho: P\ra \Z$, called a \emph{rank function for $P$}, such that
        $\rho(b)=\rho(a)+1$ whenever $a,b$ are elements in $P$ for which $a<b$
        is a covering relation in
        $P$. It is worth noting that for a fully commutative element
        $w$ in
        $W$, while we have described how to assign each element in the heap $P$
        of $w$ a
       well-defined level, the poset $P$ is not
       necessarily ranked, i.e. the level function may not be a rank function.
       For an example where $P$ is not ranked and for criteria for $P$ to be
       ranked, see \cite{Green_interval}.
      \end{remark}

      As we embed $P$ in the lattice $S\times \N$, after a point $k\in P$ ($1\le k\le q$) falls into position, it is customary to
      label the point with $s_k$ rather than $k$. Furthermore, to indicate the covering relations of the heap, we connect the point with
      edges to the highest existing points in its column and its adjacent
      columns. The resulting graph therefore recovers the Hasse diagram
      of the poset $P$. 
      For example, in Figure \ref{fig:example}, the picture on the right shows the heap of
      the element $abcabd$ in the Coxeter group whose Coxeter diagram is drawn on the
      left. Note that
      all reduced words of a fully commutative element result in an identical
      graph when we embed them in $S\times \N$, so
      we may identify the element with its embedding. For more on the lattice
      embeddings of heaps, see \cite{BJ}.

      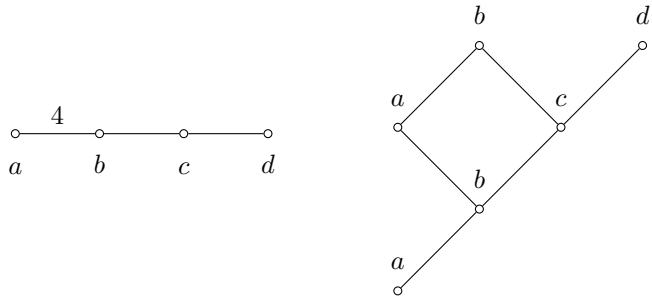
\begin{figure}
        \centering
        \begin{tikzpicture}
          \node[main node] (a) {};
          \node[main node] (b) [right=1cm of a] {};
          \node[main node] (c) [right=1cm of b] {};
          \node[main node] (d) [right=1cm of c] {};

          \node (aa) [below=0.2cm of a] {$a$};
          \node (bb) [below=0.1cm of b] {$b$};
          \node (cc) [below=0.2cm of c] {$c$};
          \node (dd) [below=0.1cm of d] {$d$};

          \path[draw]
          (b)--(c)--(d)
          (a) edge node [above] {$4$} (b);

          \node[main node] (1) [below right=2cm and 5cm of a] {};
          \node[main node] (3) [above right=1cm and 1cm of 1] {};
          \node[main node] (4) [above left=1cm and 1cm of 3] {};
          \node[main node] (5) [above right=1cm and 1cm of 3] {};
          \node[main node] (6) [above left=1cm and 1cm of 5] {};
          \node[main node] (7) [above right=1cm and 1cm of 5] {};

          \node (11) [above=0.1cm of 1] {$a$};
          \node (33) [above=0.1cm of 3] {$b$};
          \node (44) [above=0.1cm of 4] {$a$};
          \node (55) [above=0.1cm of 5] {$c$};
          \node (66) [above=0.1cm of 6] {$b$};
          \node (77) [above=0.1cm of 7] {$d$};

          \path[draw]
          (1)--(3)--(4)--(6)--(5)--(7)
          (3)--(5);

        \end{tikzpicture}

        \caption{Lattice embedding of a heap.}
        \label{fig:example}
      \end{figure}
    Note that the criterion from
      Proposition \ref{fc criterion} can now be translated as follows.
      \begin{prop}
        \label{fc heap criterion}
        The heap $P$ of a word in $S^*$ is the heap of a fully
        commutative element in $W$ if and only if in the lattice embedding of
        $P$ in $S\times \N$,
        \begin{enumerate} 
          \item No column contains two points connected by an edge.
          \item For every pair $s,t\in S$ such that $m(s,t)\ge 3$, whenever there is a
            chain of edges connecting a sequence $s,t,s,\cdots $ of $m(s,t)$
            points, there is another chain connecting two points in this sequence.
        \end{enumerate} 
      \end{prop} 
      \noindent For example, from Figure \ref{fig:example}, we easily see that
      the element $abcabd$ is fully commutative. In particular, although the
      heap contains the chains with labels $(a,b,a,b)$ and $(b,c,b)$ with $m(a,b)$ and $m(b,c)$ letters, respectively, each of these chains contains two points connected by the other chain.

      In addition to detecting fully commutative
      elements, heaps provide a convenient tool for visualizing generalized star operations on a fully commutative element
      $w$.  
       More precisely, note that
        for any $u\in S$, $w$ admits a reduced word starting with $u$ if and
        only if $H(w)$, as a poset,  contains a minimal element with label
        $u$. Thus, for each pair $I=\{s,t\}\se S$ of generators with
        $m(s,t)\ge 3$, we may perform a lower left star operation with respect
        to $I$ on
        $w$ if and only if the following
        conditions hold:
        \begin{enumerate}
          \item the heap $H(w)$
        contains a minimal element $i$ labeled by an element in $I$,
        \item  the element $i$ is connected to an
        element $j$ in $H(w)$ labeled by the other element in $I$,  
        \item upon removal of $i$ from
        $H(w)$, the element $j$ becomes a minimal element in the resulting poset.
        \end{enumerate}
        When these conditions are met, performing the lower left star operation with respect to
        $I$ on $w$ corresponds to removing the vertex $i$ and all edges
        incident to $i$ in
        $H(w)$. 
        Similarly, we can easily detect when we can perform lower right star
        opeartions on
        $w$ via $H(w)$, by examining the maximal elements of suitable heaps. We
        shall refer back to these visualizations frequently in Section
        \ref{star arguments}. 
        Upper star operations can also be described in terms of heaps, but
        we will not need them, so we omit the descriptions. 

      The final feature of heaps that we are interested in concerns the computation of $\la$-values. Thanks to a powerful result of Shi in \cite{Shi}, heaps can sometimes be used to
      compute $\la$-values of fully commutative elements in the following fashion.

      \begin{prop}[{\cite[Theorem 3.1]{Shi}}]
        \label{a from heap}
        Let $W$ be a Weyl group or an affine Weyl group. Let $w$ be a fully
        commutative element of $W$, let $\mathcal{AC}$ be the collection of all
        antichains in the heap $H(w)$, and let $n(w)=\max(\abs{A}:A\in
        \mathcal{AC})$, where $\abs{A}$ denotes the cardinality of $A$ for each
        antichain $A\in \mathcal{AC}$. Then $\la(w)=n(w)$.
      \end{prop}

      \begin{remark}
        In \cite{Shi}, the author does not explicitly use heaps to describe the
        $\la$-values of fully commutative elements.
        Rather, he associates a directed graph ${G}(w)$ to each fully
        commutative element $w$, defines a number $n(w)$ using $G(w)$, then shows
        $\la(w)=n(w)$. However, as the author points out at the end of Section 2.2,
        ${G}(w)$ can be reformulated in terms of heaps, and it is not
        difficult to see that his definition of $n(w)$ is identical with ours.
      \end{remark}

      The equality $\la(w)=n(w)$ from Proposition \ref{a from heap} also holds in
      another situation: define a \emph{star reducible} Coxeter group to be a
      Coxeter group where each fully commutative
      element is star reducible to a product of mutually commuting generators, then
      the following holds.
      \begin{prop}
        \label{star reducible a}
        Let $W$ be a star reducible Coxeter group, and let $w\in W$ be a fully
        commutative element. Then $\la(w)=n(w)$, where $n(w)$ is defined as in
        Proposition \ref{a from heap}.
      \end{prop}

      \begin{proof}
        Suppose $w$ can be reduced to a product $w'=s_1\cdots s_k$ of $k$ mutually
        commuting generators of $W$ via a series of lower star operations. Then
        $\la(w)=\la(w')=k$ by Corollary \ref{star a} and Corollary \ref{product a},
        and $n(w')=k$ since no two elements in the heap of $w'$ are comparable. Thus,
        to show $\la(w)=n(w)$, it suffices to show that $n(w')=n(w)$. We do so below
        by showing that lower star operations preserve $n$-values of fully
        commutative elements.

        Let $x\in W$ be fully commutative, and suppose $y=x_*$ with respect to some
        lower right star operation. Since any antichain in the heap $H(y)$ is
        also one in $H(x)$, $n(x)\ge n(y)$ by the definition of $n$. Meanwhile, by
        assumption, $x$ admits a reduced word $y=s_1s_2\cdots s_q$ such that
        $y=s_{1}s_2\cdots s_{q-1}$. Note that $H(x)$ must contain an element
        $p$ such that $q$ is the unique element in $H(x)$ larger than $p$, for
        otherwise the right star operation removing $s_q$ from
        $x$ would not be possible. Now, if an antichain $A$ in $H(x)$ contains
        $q$, then $p$ is not in $A$ since $A$ is an
        antichain. Furthermore, let $a\in A\setminus \{q\}$, then $p\not\le a$ since $q$ is the unique
        element larger than $p$ in $H(x)$, and $a\not\le p$ since otherwise
        $a\le q$ by transitivity, contradicting the fact that $A$ is an
        antichain. Thus, for any antichain $A$ of $H(x)$ that has length
        $n(x)$ and contains $q$, the set $(A\setminus\{q\})\cup\{p\}$ forms an
        antichain of the same length. This new
        antichain is also an antichain in $H(y)$, therefore we have $n(y)\ge n(x)$. We have thus proved $n(y)=n(x)$, i.e. that lower right star
        operations preserve $n$-values of fully commutative elements. A similar
        argument shows that the same is true for lower left star operations, so we
        are done. 
      \end{proof}
      \noindent For more on star reducible Coxeter
      groups, including the classification of all star reducible Coxeter groups, see
      \cite{star_reducible}.

      By propositions \ref{a2 is fc}, \ref{a from heap} and \ref{star reducible a}, to show
      $\{w\in W:\la(w)=2\}$ is infinite for a Weyl group, affine Weyl group or a
      star reducible Coxeter group, it
      suffices to produce infinitely many distinct fully commutative elements,
      examine their heaps, then
      use the antichain characterization to verify that the elements have
      $\la$-value 2. We will repeatedly use this strategy in Section
      \ref{first lemmas}.

      \section{Proof of Theorem \ref{first theorem}: Sufficiency of the diagram
    criteria}
      \label{sufficient}
      Let $W$ be an irreducible Coxeter group with Coxeter diagram $G$. We prove the
      ``if'' directions of the two parts
      of Theorem \ref{first theorem} in this section, i.e. we show that $W$ is $\la(2)$-finite
      if $G$ is as described in the theorem.

      \subsection{Case 1. $G$ contains a cycle} We first prove the ``if''
      direction of Theorem \ref{first theorem}.(1). Since $W$ is certainly
      irreducible and $G$ certainly contains a cycle when $G$ is a complete
      graph with 3 or more vertices, it suffices to prove the following.       
      \begin{prop}
        \label{sufficient cycle}
        If $G$ is a complete graph, then  $W$ is $\la(2)$-finite.
      \end{prop}
      \begin{proof}
        We claim that $W$ actually contains no element of $\la$-value 2 if $G$
        is complete. To see this, suppose $a(w)=2$ for some $w\in W$.
        Then $w$ is fully commutative by Proposition \ref{a2 is
        fc}. But as $G$ is complete, no two elements of the generating set
        of $W$ commute, therefore an element in $W$ is fully commutative if and only
        if it has a unique reduced word. Proposition \ref{subregular} then
        implies that $\la(w)\le 1$, a contradiction.
      \end{proof}

      \subsection{Case 2. $G$ is acyclic}
      \label{case 2} 
      We now prove
      the ``if'' part of Theorem \ref{first theorem}.(2),  which is restated
      below.  
  \begin{prop}
        \label{sufficient prop}
        Let $W$ be an irreducible Coxeter group with Coxeter diagram $G$. 
        If $G$ is one of the graphs shown in Figure \ref{fig:finite E}, 
        then $W$ is $\la(2)$-finite.
      \end{prop} 
      It turns out that when $G$ is any graph from Figure \ref{fig:finite E} other than
      $E_{q,r}$ where $\min(q,r)\ge 3$, we may use two
      key results not yet stated in the paper to prove that $W$ is
      $\la(2)$-finite. The first of these results is the following
      classification of Stembridge. 

      \begin{prop}[{\cite[Theorem 5.1]{FC}}]
        \label{fc finite}
        An irreducible Coxeter group has finitely many fully commutative elements if
        and only if its Coxeter diagram is of the form $A_n (n\ge 1), B_n (n\ge 2),
        D_n (n\ge 4), E_n (n\ge 6), F_n (n\ge 4), H_n (n\ge 3)$ or $I_2(m)
        (5\le m\le\infty)$.
      \end{prop}
      \noindent Recall that the graphs of type $D_n$ and  $E_n$ here are special
      cases of the graphs $E_{q,r}$ from Figure \ref{fig:finite E} (see Remark
      \ref{DE}).

      The second external result was established by D. Ernst in \cite{Ernst}. 

      \begin{prop}[{\cite[Corollary 5.16]{Ernst}}]
        \label{Ernst}
        Let $W$ be the affine Coxeter group of type $\tilde C_n$ for some $n\ge
        5$,
        i.e. suppose its Coxeter diagram is of the form $\tilde C_n$ from Figure
        \ref{fig:finite E}. Then $W$ is $\la(2)$-finite.
      \end{prop}
      We now deal with the case where $G$ is of the form $E_{q,r}$ where
      $\min(q,r)\ge 3$. We will prove that $W$ is $\la(2)$-finite in this case
      in Proposition \ref{ae7 case}, after we prove a series of lemmas.
      We will then combine the external results and Proposition \ref{ae7 case}
      to finish the proof of Proposition
      \ref{sufficient prop}.

      Throughout the following four lemmas, let $w$ be a fully commutative element in
      $W$.       For any $s\in S$ that labels at least two elements in $H(w)$, define an \emph{open $s$-interval}  in $H(w)$ to be an interval
      $(i,j)=\{k\in H(w):i<k<j\}$ where $i$ and $j$ are consecutive elements
      labelled by $s$; similarly, define a \emph{closed $s$-interval} to be an
      interval
      of the form $[i,j]=\{k\in H(w):i\le k\le j\}$ where $i$ and $j$ are
      consecutive elements labelled by $s$.

      \begin{lemma}
Suppose $G$ is of type $A$, and let $s\in S$. Every open $s$-interval in
        $H(w)$ contains exactly two elements whose labels are adjacent to
        $s$ in $G$, and the labels of these elements are distinct. In
        particular, if $s$ is an endpoint of $G$, then $w$ contains at most one
        occurrence of $s$.
        \label{type A sandwich}
      \end{lemma}

      \begin{proof}
        This is well-known; see, for example, Remark 3.3.7 of
        \cite{Green_interval}.
      \end{proof}

      \begin{lemma}
        Suppose $G$ is the Coxeter diagram shown in Figure \ref{ae7 graph}, and suppose
        $\la(w)\le 2$.   Then any open $d$-interval in $H(w)$ contains exactly two
        elements with labels from the set $\{c,e,h\}$, and their labels are
        distinct.

        \begin{figure}
          \begin{tikzpicture}

      \node[main node] (31) {}; 
      \node[main node] (32) [right=1cm of 31] {};
      \node[main node] (33) [right=1cm of 32] {};
      \node[main node] (34) [right=1cm of 33] {};
      \node[main node] (38) [right=1cm of 34] {};
      \node[main node] (39) [right=1cm of 38] {};
      \node[main node] (40) [right=1cm of 39] {}; 
      \node[main node] (41) [above=0.8cm of 34] {};

      \node (a) [below=0.2cm of 31] {a};
      \node (b) [below=0.1cm of 32] {b};
      \node (c) [below=0.2cm of 33] {c};
      \node (d) [below=0.1cm of 34] {d};
      \node (e) [below=0.2cm of 38] {e};
      \node (f) [below=0.1cm of 39] {f};
      \node (g) [below=0.2cm of 40] {g};
      \node (h) [above=0.1cm of 41] {h};
      
      \path[draw]
      (31)--(32)--(33)--(34)--(38)--(39)--(40)
      (34)--(41);
    \end{tikzpicture}
    \caption{}
    \label{ae7 graph}
  \end{figure}
  \label{ae7}
                \end{lemma}

                \begin{proof}
                  Let $I$ be an open $d$-interval.  Deleting $d$ from $G$ produces a union of three subgraphs of
                  type $A$ in which $c,e$ and $h$ appear as endpoints,
                  therefore Lemma \ref{type A sandwich} implies that each of $c,e$ and
                  $h$ can appear at most once (as the label of an element) in
                  $I$. Since $w$ is fully commutative, at least two of them
                  must appear in $I$. Finally, since $G$ is the Coxeter diagram
                 $\tilde E_7$, an affine Weyl group of type $E$, $c,e,h$ cannot all appear in
                  $I$ because otherwise the corresponding elements would form an
                  antichain of length 3 and we would have $\la(w)=n(w)\ge 3$ by
                  Proposition \ref{a from heap}. It follows that $I$ contains exactly two
                  elements with distinct labels from $\{c,e,h\}$.
                \end{proof}

                \begin{lemma}
                  Let $G$ and $w$ be as in Lemma \ref{ae7}. Then any open $h$-interval in
                  $H(w)$ must contain precisely two occurrences of $d$.
                  \label{d between h}
                \end{lemma}

                \begin{proof}
                  Let $I$ be an open $h$-interval.
                Since $w$ is fully commutative and $d$ is the only vertex
                adjacent to $h$ in $G$, $I$ must contain at least two
                occurrences of $d$, so it suffices to show that $I$ cannot
                contain three or more occurrences of $d$. 

                For a contradiction, suppose $I$ contains three elements
                $d_1,d_2,d_3$ with label $d$. By definition, all elements in $I$
                are labelled by vertices from the subgraph of type $A$ induced
                by $a,b,\cdots g$, therefore the interval $(d_1,d_2)$ contains
                exactly one element with label $c$ and exactly one element
                with label $e$ by Lemma \ref{type A sandwich}; call them $c_1$ and
                $e_1$, respectively. Similarly, $(d_2,d_3)$ contains unique
                elements $c_2$ and $e_2$ with labels $c$ and $e$, respectively.
                Thus, the interval $(d_1,d_3)$ contains the sequence 
                \[
                d_1, c_1,e_1, d_2, c_2,e_2,d_3
                \]
                in weakly increasing order.

             Now consider the interval $J=(c_1,c_2)$. Since it contains exactly
             one occurrence of $d$ and $w$ is fully commutative,
             $J$ contains at least
                one occurrence of $b$. Moreover, since deletion of $c$ from
                $G$ leaves $b$ as an endpoint on a subgraph of type $A$,
                the appearance of $b$ in $J$ must be unique by Lemma \ref{type
                A sandwich}. Similarly, by considering the interval
                $J'=(e_1,e_2)$, we may conclude that $J'$ contains a unique
                occurrence of $f$. But then the elements with labels $b,d,f$
                in $(d_1,d_3)$ form an antichain of length $3$ in $H(w)$,
                therefore $\la(w)= n(w)\ge 3$ since $G$ is of type $\tilde E$.
                This contradicts our assumption that $\la(w)\le 2$, therefore
                $I$ contains precisely two occurrences of $d$. 
              \end{proof}

              \begin{lemma}
                Let $G$ and $w$ be as in Lemma \ref{ae7}. Then $w$ cannot
                contain three or more occurrences of $h$.
                \label{no three h}
              \end{lemma}
              \begin{proof}
                Suppose for a contradiction that $H(w)$ contains three
                consecutive elements $h_1, h_2$ and $h_3$ with label $h$. By
                Lemma \ref{d between h}, there are precisely two elements
                $d_1, d_2$ with label $d$ in $(h_1,h_2)$ and  two elements
                $d_3,d_4$ with label $d$ in $(h_2,h_3)$. Moreover, by Lemma
                \ref{ae7}, $(d_1,d_2)$ and $(d_3,d_4)$ each contains exactly
                one occurrence each of $c$ and $e$, and $(d_2,d_3)$ contains
                one occurrence of $c$ or $e$. Without loss of generality,
                suppose $(d_2,d_3)$ contains an element labelled by $c$. Then the
                interval $[h_1,h_3]$ contains the sequence
                \[
                h_1,d_1,c_1,e_1,d_2,c_2, h_2, d_3,c_3, e_2,d_4, h_3
                \]
                in weakly increasing order. In this sequence, each $c_i$ is
                labelled by $c$, each $e_i$ is labelled by
                $e$, and all elements in $[h_1,h_3]$ with labels $c,d,e$ or
                $h$ have been listed.

                Arguing as in the proof of  Lemma \ref{d between h}, we see
                that $(c_1,c_2)$ must contain a unique element, say $b_1$, with
                label $b$, and $(c_2,c_3)$ must contain a unique element, say
                $b_2$, with label $b$.  Moreover, any two occurrences of
                $b$ must be separated by an occurrence of $c$, therefore
                $b_1$ and $b_2$ are consecutive elements with label $b$. But
                then there must be an element, say $a_1$, with label $a$ in
                $(b_1,b_2)$. The elements $a_1, c_2,h_2$ now form an antichain
                of length $3$ in $H(w)$, therefore $\la(w)\ge 3$, contradicting
                the assumption that $\la(w)\le 2$.
              \end{proof}

              \begin{prop}
                \label{ae7 case}
                Let $G$ be of the form $E_{q,r}$ from Figure \ref{fig:finite E}, and suppose
                $\min(q,r)\ge 3$. Then $W$ is $\la(2)$-finite.
              \end{prop}
              \begin{proof}
            Denote the top vertex on the shortest branch of $G$ by $s$, and let
            $W'$ be the Coxeter group generated by $S\setminus \{s\}$. Any
            element $w\in W$ with $\la(w)=2$ can be written in the form $w=w_1
            sw_2s\cdots sw_n$ for some $n\ge 0$ and $w_1, \cdots, w_n\in W'$.
            By Lemma \ref{no three h}, we must have $n\le 3$ if $\la(w)=2$. Since $W'$ is of type $A$ and thus finite, this implies that $W$ is $\la(2)$-finite. 
                \end{proof}

                We can now prove Proposition \ref{sufficient
                prop}, i.e. the ``if'' direction of Theorem  \ref{first
                theorem}.(2).

                \begin{proof}[Proof of Proposition \ref{sufficient prop}]
        Recall that any element of $\la$-value 2 is necessarily fully commutative by
        Proposition \ref{a2 is fc}. Thus, Proposition \ref{fc finite} implies that
        $W$ is $\la(2)$-finite if $G$ is of type $A, B, E_{1,r}, E_{2,r}, F, H$ or
        $I_2(m)$ where $5\le m<\infty$ from
        Figure \ref{fig:finite E}. If $G=I_2(\infty)$, $W$ is $\la(2)$-finite
        by Proposition \ref{sufficient cycle}. If $G$ is of type
        $\tilde C$,
        $W$ is $\la(2)$-finite by Proposition
        \ref{Ernst}. Finally, Proposition \ref{ae7 case} says that $W$ is
        $\la(2)$-finite if $G$ is of the form $E_{q,r}$ where
        $q,r\ge 3$. This completes the proof.
      \end{proof}

      \section{Proof of Theorem \ref{first theorem}: Necessity of the diagram
    criteria}
    \label{sec:necessity}
     Let $W$ be an irreducible Coxeter group with Coxeter diagram $G$. We now
     prove the ``only if'' direction of Theorem \ref{first theorem}.
To do so, we first prove a series of lemmas that
      each says that $W$ is $\la(2)$-infinite if $G$ contains a
      certain subgraph. We then argue in the
      last subsection that
      in order for $G$ not to contain these subgraphs, it has to be a graph from Figure \ref{fig:finite E}.

We shall call the elements of $\la$-value 2 in our lemmas
      \emph{witnesses}. Based on the method we use to prove that the witnesses have $\la$-value
      2, we will group our lemmas into three subsections.

      By Proposition \ref{a local}, to show that $W$ is $\la(2)$-infinite when $G$ contains a certain subgraph
      $G'$, it suffices to find infinitely many witnesses of $\la$-value 2 in the Coxeter group with $G'$ as its Coxeter
      diagram. We will use this fact without comment throughout the rest of the
      paper.

    \subsection{Lemmas with heap arguments}
      \label{first lemmas}
     For our first set of lemmas, the proofs that the
      witnesses have $\la$-value 2 will rely only on propositions
      \ref{a from heap} and \ref{star reducible a} from Section \ref{fc}. 
      In particular, no star operations will be involved in the arguments.


      \begin{lemma}
        \label{affine B}
        If $G$ contains a subgraph of the form
        
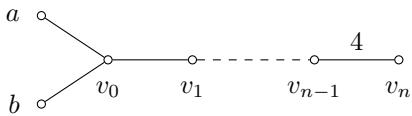
\begin{center}
          \begin{tikzpicture}
            \node[main node] (1) {};
            \node[main node] (8) [above left = 0.5cm and 0.8cm of 1] {};
            \node[main node] (9) [below left = 0.5cm and 0.8cm of 1] {};
            \node[main node] (2) [right=1cm of 1] {};
            \node[main node] (5) [right=1.5cm of 2] {};
            \node[main node] (6) [right=1cm of 5] {};

            \node (88) [left=0.1cm of 8] {$a$};
            \node (99) [left=0.1cm of 9] {$b$};
            \node (11) [below=0.1cm of 1] {$v_0$};
            \node (22) [below=0.1cm of 2] {$v_1$};
            \node (55) [below=0.1cm of 5] {$v_{n-1}$};
            \node (66) [below=0.1cm of 6] {$v_n$};

            \path[draw]
            (8) edge node {} (1)
            (9) edge node {} (1)
            (1) edge node {} (2)
            (5) edge node [above] {$4$} (6);
            \path[draw,dashed]
            (2)--(5);
          \end{tikzpicture}
        \end{center}
\noindent where $n\ge 1$ and all edges other than $\{v_{n-1},v_n\}$
        have weight 3. Then $W$ is $\la(2)$-infinite. 

             \end{lemma}

      \begin{proof}
Let \[
    w_k=(abv_0v_1\cdots
  v_{n-1}v_nv_{n-1}\cdots v_1v_0)^k.
\]
for $k\in \Z_{\ge 1}$.
The heap of $w_k$ is shown below, where the dashed rectangles correspond to the parenthesized
        expression in $w_k$ and are repeated $k$ times.
  
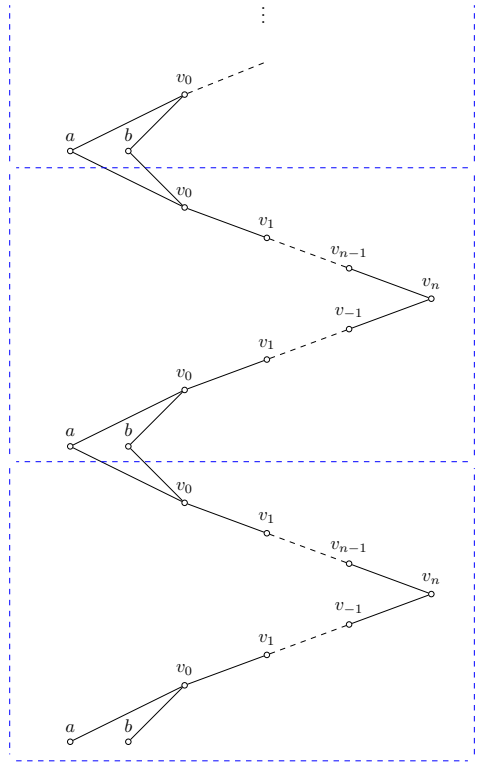
\begin{center}
          \resizebox{0.5\textwidth}{!}{
            \begin{tikzpicture} 
              \node[main node] (1) {};
              \node[main node] (0) [below left = 1cm and 1cm of 1] {};
              \node (00) [above=0.01cm of 0] {$b$};
              \node[main node] (000) [left=1cm of 0] {};
              \node (0000) [above=0.01cm of 000] {$a$};
              \node[main node] (2) [above right = 0.5cm and 1.5cm of 1] {};
              \node[main node] (3) [above right = 0.5cm and 1.5cm of 2] {};
              \node[main node] (4) [above right = 0.5cm and 1.5cm of 3] {};
              \node[main node] (5) [above left = 0.5cm and 1.5cm of 4] {};
              \node[main node] (6) [above left = 0.5cm and 1.5cm of 5] {};
              \node[main node] (7) [above left = 0.5cm and 1.5cm of 6] {};
              \node[main node] (9) [above left = 1cm and 1cm of 7] {};
              \node[main node] (8) [left = 1cm of 9] {};
              \node[main node] (b1) [above right = 1cm and 1cm of 9] {};
              \node[main node] (b2) [above right = 0.5cm and 1.5cm of b1] {};
              \node[main node] (b3) [above right = 0.5cm and 1.5cm of b2] {};
              \node[main node] (b4) [above right = 0.5cm and 1.5cm of b3] {};
              \node[main node] (b5) [above left = 0.5cm and 1.5cm of b4] {};
              \node[main node] (b6) [above left = 0.5cm and 1.5cm of b5] {};
              \node[main node] (b7) [above left = 0.5cm and 1.5cm of b6] {};
              \node[main node] (b9) [above left = 1cm and 1cm of b7] {};
              \node[main node] (b8) [left = 1cm of b9] {};
              \node[main node] (bb1) [above right = 1cm and 1cm of b9] {};
              \node (bb2) [above right=0.5cm and 1.5cm of bb1] {};

              \node (c1) [above=0.01cm of 1] {$v_0$};
              \node (c2) [above=0.01cm of 2] {$v_1$};
              \node (c3) [above=0.01cm of 3] {$v_{-1}$};
              \node (c4) [above=0.01cm of 4] {$v_n$};
              \node (c5) [above=0.01cm of 5] {$v_{n-1}$};
              \node (c6) [above=0.01cm of 6] {$v_1$};
              \node (c7) [above=0.01cm of 7] {$v_0$};
              \node (c8) [above=0.01cm of 8] {$a$};
              \node (c9) [above=0.01cm of 9] {$b$};
              \node (bc1) [above=0.01cm of b1] {$v_0$};
              \node (bc2) [above=0.01cm of b2] {$v_1$};
              \node (bc3) [above=0.01cm of b3] {$v_{-1}$};
              \node (bc4) [above=0.01cm of b4] {$v_n$};
              \node (bc5) [above=0.01cm of b5] {$v_{n-1}$};
              \node (bc6) [above=0.01cm of b6] {$v_1$};
              \node (bc7) [above=0.01cm of b7] {$v_0$};
              \node (bc8) [above=0.01cm of b8] {$a$};
              \node (bc9) [above=0.01cm of b9] {$b$};
              \node (bbb1) [above=0.01cm of bb1] {$v_0$};

              \path[draw]
              (000)--(1)
              (0)--(1)--(2)
              (3)--(4)--(5)
              (6)--(7)--(9)--(b1)--(b2)
              (b3)--(b4)--(b5)
              (b6)--(b7)--(b9)--(bb1)
              (7)--(8)--(b1)
              (b7)--(b8)--(bb1);

              \path[draw,dashed]
              (2)--(3)
              (5)--(6)
              (b2)--(b3)
              (b5)--(b6)
              (bb1)--(bb2);

              \node (r1) [below left = 0.2cm and 1cm of 000] {};
              \node (r2) [below right = 0.2cm and 6.5cm of 0] {};
              \node (r3) [above = 5.5cm of r1] {};
              \node (r4) [above = 5.5cm of r2] {};
              \node (r5) [above = 5.4cm of r3] {};
              \node (r6) [above = 5.4cm of r4] {};
              \node (r7) [above = 3cm of r5] {};
              \node (r8) [above = 3cm of r6] {};
              \node (r9) [above right=1.2cm and 1.3 cm of bb1] {$\vdots$};

              \path[draw,dashed,blue]
              (r1)--(r2)--(r4)--(r3)--(r1)
              (r4)--(r6)--(r5)--(r3)
              (r5)--(r7)
              (r6)--(r8);
            \end{tikzpicture}
          }
        \end{center}

        For each $k\ge 1$, it is clear from the figure that $w_k$ is
  reduced and fully commutative by Proposition \ref{fc heap criterion}. Furthermore,
        observe that any two elements from consecutive levels of $H(w_k)$ are
        comparable, hence any antichain of maximal length in $H(w_k)$ must contain
        exactly the two elements labelled by $a$ and $b$ on a same level,
        therefore $n(w_k)=2$. 
  Since the subgraph in question is the Coxeter
  diagram of an affine Weyl group of type $B$, it follows from
  Proposition
  \ref{a from heap} that $\la(w_k)=2$ for all $k\ge 1$, therefore $W$ is
  $\la(2)$-infinite.
      \end{proof}

      \begin{lemma}
        \label{affine C3}
       If $G$ contains a subgraph of the following form, then $W$ is $\la(2)$-infinite. 
        
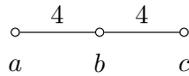
\begin{center}
\begin{tikzpicture}
            \node[main node] (1) {};
            \node[main node] (2) [right = 1cm of 1] {};
            \node[main node] (3) [right = 1cm of 2] {};

            \node (11) [below = 0.18cm of 1] {$a$};
            \node (22) [below = 0.1cm of 2] {$b$};
            \node (33) [below = 0.18cm of 3] {$c$};

            \path[draw]
            (1) edge node [above] {$4$} (2)
            (2) edge node [above] {$4$} (3);
          \end{tikzpicture}
        \end{center}

      \end{lemma}

      \begin{proof}
        Let $w_k=(acb)^k$ for $k\in \Z_{\ge 1}$. The heap of $w_k$ is shown
   below.
           \begin{center}
          \resizebox{0.4\textwidth}{!}
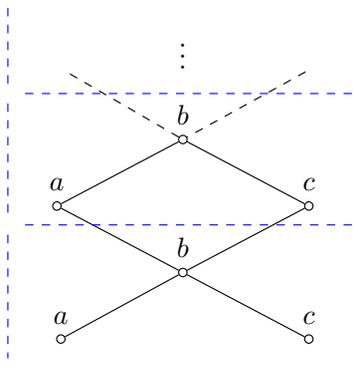
{
            \begin{tikzpicture}
              \node[main node] (1) {};
              \node[main node] (1a) [left=3cm of 1] {};
              \node[main node] (2) [above left=0.75cm and 1.5cm of 1] {};
              \node[main node] (3) [above left=0.75cm and 1.5cm of 2] {};
              \node[main node] (4) [above right=0.75cm and 1.5cm of 2] {};
              \node[main node] (5) [above right=0.75cm and 1.5cm of 3] {};
              \node (6) [above right=0.75cm and 1.5cm of 5] {};
              \node (2a) [left=3cm of 6] {};

              \node (1a1a) [above=0.01cm of 1a] {$a$};
              \node (11) [above=0.01cm of 1] {$c$};
              \node (22) [above=0.01cm of 2] {$b$};
              \node (33) [above=0.01cm of 3] {$a$};
              \node (44) [above=0.01cm of 4] {$c$};
              \node (55) [above=0.01cm of 5] {$b$};

              \path[draw]
              (1)--(2)--(3)--(5)
              (1a)--(2)--(4)--(5);

              \path[draw,dashed]
              (6)--(5)--(2a);

              \node (r1) [below left = 0.2cm and 0.5cm of 1a] {};
              \node (r2) [below right = 0.2cm and 0.5cm of 1] {};
              \node (r3) [above = 1.55cm of r1] {};
              \node (r4) [above = 1.55cm of r2] {};
              \node (r5) [above = 1.4cm of r3] {};
              \node (r6) [above = 1.4cm of r4] {};
              \node (r7) [above = 1cm of r5] {};
              \node (r8) [above = 1cm of r6] {};
              \node (r9) [above = 0.7cm of 5] {$\vdots$};

              \path[draw,dashed,blue]
              (r1)--(r2)--(r4)--(r3)--(r1)
              (r4)--(r6)--(r5)--(r3)
              (r5)--(r7)
              (r6)--(r8);
            \end{tikzpicture}
          }
        \end{center}
       
        As in the previous lemma, it is clear that $w_k$ is reduced and fully commutative by Proposition \ref{fc heap
        criterion} and that $n(w_k)=2$ for each $k$. Since the subgraph in
        question is the Coxeter diagram of $\tilde{C}_3$,
        an affine Weyl group of type $C$, Proposition \ref{a from heap}
        implies that $\la(w_k)=2$ for all $k\ge 1$, therefore $W$ is
        $\la(2)$-infinite.  
      \end{proof}


      \begin{lemma}
        \label{affine C4}
        If $G$ contains a subgraph of the following form, then $W$ is $\la(2)$-infinite. 
        
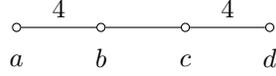
\begin{center}
          \begin{tikzpicture}
            \node[main node] (1) {};
            \node[main node] (2) [right = 1cm of 1] {};
            \node[main node] (a) [right = 1cm of 2] {};
            \node[main node] (3) [right = 1cm of a] {};

            \node (11) [below = 0.18cm of 1] {$a$};
            \node (22) [below = 0.1cm of 2] {$b$};
            \node (33) [below = 0.18cm of a] {$c$};
            \node (44) [below = 0.1cm of 3] {$d$};

            \path[draw]
            (1) edge node [above] {$4$} (2)
            (a) edge node [above] {$4$} (3);

            \path[draw]
            (2)--(a);
\label{ac4}
          \end{tikzpicture}
        \end{center}

      \end{lemma}

      \begin{proof}
        The proof is similar to that of Lemma \ref{affine C3}: the graph in
        question is still of type $\tilde C_n$, and the elements $w_k=(acbd)^k$
        where
        $k\in \Z_{\ge 1}$ now suffice as our witnesses. The fact that
        $\la(w_k)=n(w_k)=2$ for each $k\ge 1$ is evident from the heap of
        $w_k$, which is shown below.
        \begin{center}
          \resizebox{0.4\textwidth}{!}
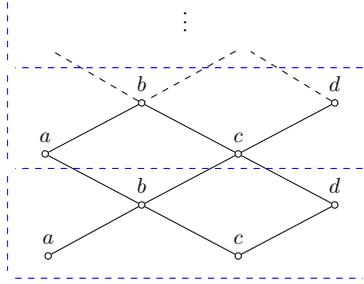
{
            \begin{tikzpicture}
              \node[main node] (1) {};
              \node[main node] (1a) [left=3cm of 1] {};
              \node[main node] (2) [above left=0.75cm and 1.5cm of 1] {};
              \node[main node] (d) [above right=0.75cm and 1.5cm of 1] {};
              \node[main node] (3) [above left=0.75cm and 1.5cm of 2] {};
              \node[main node] (4) [above right=0.75cm and 1.5cm of 2] {};
              \node[main node] (5) [above right=0.75cm and 1.5cm of 3] {};
              \node[main node] (ddd) [above right=0.75cm and 1.5cm of 4] {};
              \node (6) [above right=0.75cm and 1.5cm of 5] {};
              \node (2a) [left=3cm of 6] {};

              \node (1a1a) [above=0.01cm of 1a] {$a$};
              \node (11) [above=0.01cm of 1] {$c$};
              \node (22) [above=0.01cm of 2] {$b$};
              \node (dd) [above=0.01cm of d] {$d$};
              \node (33) [above=0.01cm of 3] {$a$};
              \node (44) [above=0.01cm of 4] {$c$};
              \node (55) [above=0.01cm of 5] {$b$};
              \node (dddd) [above=0.01cm of ddd] {$d$};

              \path[draw]
              (1)--(2)--(3)--(5)
              (1)--(d)--(4)--(ddd)
              (1a)--(2)--(4)--(5);

              \path[draw,dashed]
              (6)--(5)--(2a)
              (ddd)--(6);

              \node (r1) [below left = 0.2cm and 0.5cm of 1a] {};
              \node (r2) [below right = 0.2cm and 2cm of 1] {};
              \node (r3) [above = 1.55cm of r1] {};
              \node (r4) [above = 1.55cm of r2] {};
              \node (r5) [above = 1.4cm of r3] {};
              \node (r6) [above = 1.4cm of r4] {};
              \node (r7) [above = 1cm of r5] {};
              \node (r8) [above = 1cm of r6] {};
              \node (r9) [above right= 1cm and 0.5cm of 5] {$\vdots$};

              \path[draw,dashed,blue]
              (r1)--(r2)--(r4)--(r3)--(r1)
              (r4)--(r6)--(r5)--(r3)
              (r5)--(r7)
              (r6)--(r8);
            \end{tikzpicture}
          }
        \end{center}
      \end{proof}
      \begin{lemma}
        \label{two branches}
        If $G$ contains a subgraph of the form

        
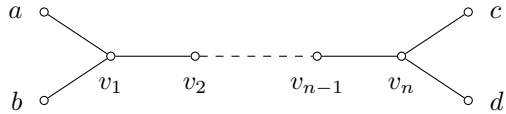
\begin{center}
          \begin{tikzpicture}
            \node[main node] (1) {};
            \node[main node] (8) [above left = 0.5cm and 0.8cm of 1] {};
            \node[main node] (9) [below left = 0.5cm and 0.8cm of 1] {};
            \node[main node] (2) [right=1cm of 1] {};
            \node[main node] (5) [right=1.5cm of 2] {};
            \node[main node] (6) [right=1cm of 5] {};
            \node[main node] (10) [above right = 0.5cm and 0.8cm of 6] {};
            \node[main node] (12) [below right = 0.5cm and 0.8cm of 6] {};

            \node (88) [left=0.1cm of 8] {$a$};
            \node (99) [left=0.1cm of 9] {$b$};
            \node (11) [below=0.1cm of 1] {$v_1$};
            \node (22) [below=0.1cm of 2] {$v_2$};
            \node (55) [below=0.1cm of 5] {$v_{n-1}$};
            \node (66) [below=0.1cm of 6] {$v_n$};
            \node (1010) [right=0.1cm of 10] {$c$};
            \node (1212) [right=0.1cm of 12] {$d$};

            \path[draw]
            (8) edge node {} (1)
            (9) edge node {} (1)
            (1) edge node {} (2)
            (5) edge node {} (6)
            (6) edge node {} (10)
            (6) edge node {} (12);
            \path[draw,dashed]
            (2)--(5);
          \label{tb}
          \end{tikzpicture}
        \end{center}
        \noindent where $n\in \Z_{\ge 1}$ and all edges have weight 3, then $W$ is $\la(2)$-infinite. 
 
      \end{lemma}
      \begin{proof}
        Consider the elements 
        \[
          w_k=ab(v_1v_2\cdots v_ncdv_n\cdots v_2v_1ab)^k
        \]
        for $k\in \Z_{\ge 0}$. The heap of $w_k$ is shown below.        
        \begin{center}
          \resizebox{0.45\textwidth}{!}
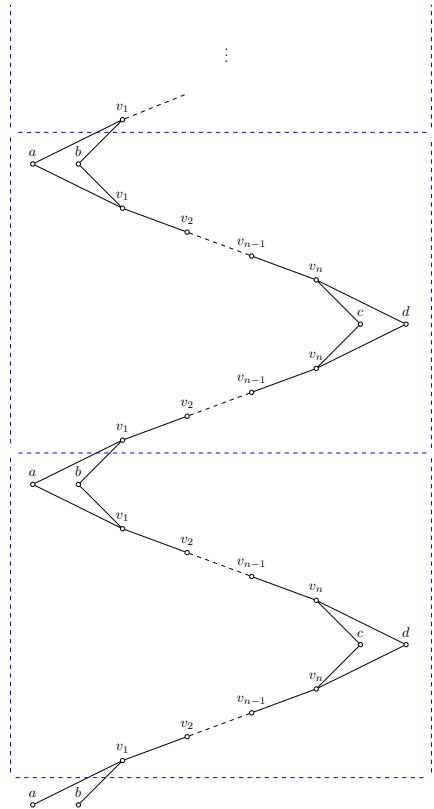
{
            \begin{tikzpicture}
              \node[main node] (1) {};
              \node[main node] (2) [left  = 1cm of 1] {};
              \node[main node] (3) [above right = 1cm and 1cm of 1] {};
              \node[main node] (4) [above right = 0.5cm and 1.5cm of 3] {};
              \node[main node] (5) [above right = 0.5cm and 1.5cm of 4] {};
              \node[main node] (6) [above right= 0.5cm and 1.5cm of 5] {};
              \node[main node] (7) [above right = 1cm and 1cm of 6] {};
              \node[main node] (8) [right = 1cm of 7] {};
              \node[main node] (9) [above left = 1cm and 1cm of 7] {};
              \node[main node] (10) [above left = 0.5cm and 1.5cm of 9] {};
              \node[main node] (11) [above left = 0.5cm and 1.5cm of 10] {};
              \node[main node] (12) [above left = 0.5cm and 1.5cm of 11] {};
              \node[main node] (b1) [above left = 1cm and 1cm of 12] {};
              \node[main node] (b2) [left  = 1cm of b1] {};
              \node[main node] (b3) [above right = 1cm and 1cm of b1] {};
              \node[main node] (b4) [above right = 0.5cm and 1.5cm of b3] {};
              \node[main node] (b5) [above right = 0.5cm and 1.5cm of b4] {};
              \node[main node] (b6) [above right= 0.5cm and 1.5cm of b5] {};
              \node[main node] (b7) [above right = 1cm and 1cm of b6] {};
              \node[main node] (b8) [right = 1cm of b7] {};
              \node[main node] (b9) [above left = 1cm and 1cm of b7] {};
              \node[main node] (b10) [above left = 0.5cm and 1.5cm of b9] {};
              \node[main node] (b11) [above left = 0.5cm and 1.5cm of b10] {};
              \node[main node] (b12) [above left = 0.5cm and 1.5cm of b11] {};
              \node[main node] (bb1) [above left = 1cm and 1cm of b12] {};
              \node[main node] (bb2) [left  = 1cm of bb1] {};
              \node[main node] (bb3) [above right = 1cm and 1cm of bb1] {};
              \node (bb4) [above right = 0.5cm and 1.5cm of bb3] {};

              \node (c1) [above=0.01cm of 1] {$b$};
              \node (c2) [above=0.01cm of 2] {$a$};
              \node (c3) [above=0.01cm of 3] {$v_{1}$};
              \node (c4) [above=0.01cm of 4] {$v_2$};
              \node (c5) [above=0.01cm of 5] {$v_{n-1}$};
              \node (c6) [above=0.01cm of 6] {$v_n$};
              \node (c7) [above=0.01cm of 7] {$c$};
              \node (c8) [above=0.01cm of 8] {$d$};
              \node (c9) [above=0.01cm of 9] {$v_n$};
              \node (c10) [above=0.01cm of 10] {$v_{n-1}$};
              \node (c11) [above=0.01cm of 11] {$v_2$};
              \node (c12) [above=0.01cm of 12] {$v_1$};
              \node (bc1) [above=0.01cm of b1] {$b$};
              \node (bc2) [above=0.01cm of b2] {$a$};
              \node (bc3) [above=0.01cm of b3] {$v_{1}$};
              \node (bc4) [above=0.01cm of b4] {$v_2$};
              \node (bc5) [above=0.01cm of b5] {$v_{n-1}$};
              \node (bc6) [above=0.01cm of b6] {$v_n$};
              \node (bc7) [above=0.01cm of b7] {$c$};
              \node (bc8) [above=0.01cm of b8] {$d$};
              \node (bc9) [above=0.01cm of b9] {$v_n$};
              \node (bc10) [above=0.01cm of b10] {$v_{n-1}$};
              \node (bc11) [above=0.01cm of b11] {$v_2$};
              \node (bc12) [above=0.01cm of b12] {$v_1$};
              \node (bbc1) [above=0.01cm of bb1] {$b$};
              \node (bbc2) [above=0.01cm of bb2] {$a$};
              \node (bbc3) [above=0.01cm of bb3] {$v_1$};

              \path[draw]
              (1)--(3)
              (2)--(3)--(4)
              (5)--(6)--(7)--(9)--(10)
              (6)--(8)--(9)
              (11)--(12)--(b1)
              (12)--(b2)
              (b1)--(b3)
              (b2)--(b3)--(b4)
              (b5)--(b6)--(b7)--(b9)--(b10)
              (b6)--(b8)--(b9)
              (b11)--(b12)--(bb1)
              (b12)--(bb2)
              (bb1)--(bb3)
              (bb2)--(bb3);

              \path[draw,dashed]
              (4)--(5)
              (10)--(11)
              (b4)--(b5)
              (b10)--(b11)
              (bb3)--(bb4);

              \node (r1) [above left = 0.5cm and 1.5cm of 1] {};
              \node (r2) [above right = 0.5cm and 8.5cm of 1] {};
              \node (r3) [above = 7.7cm of r1] {};
              \node (r4) [above = 7.7cm of r2] {};
              \node (r5) [above = 7.6cm of r3] {};
              \node (r6) [above = 7.6cm of r4] {};
              \node (r7) [above = 3cm of r5] {};
              \node (r8) [above = 3cm of r6] {};
              \node (r9) [above right=4cm and 0.75 cm of b11] {$\vdots$};

              \path[draw,dashed,blue]
              (r1)--(r2)--(r4)--(r3)--(r1)
              (r4)--(r6)--(r5)--(r3)
              (r5)--(r7)
              (r6)--(r8);
            \end{tikzpicture}
          }
        \end{center}

        From the figure, it is clear that $w_k$ is reduced
        and fully commutative for each $k\ge 0$ by Proposition \ref{fc heap
criterion}. As in Lemma \ref{affine C3}, it is also clear that $n(w_k)=2$ for
all $k\ge 0$. Since the subgraph in question is the Coxeter diagram of an affine Weyl group of type
        $D$, Proposition \ref{a from heap} implies that $\la(w_k)=2$
        for all $k\ge 2$ , therefore $W$ is $\la(2)$-infinite.  
      \end{proof}

      \begin{lemma}
        \label{affine E6}
        If $G$ contains a subgraph of the following form, then $W$ is $\la(2)$-infinite. 
 
        
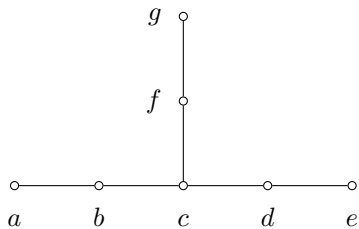
\begin{center}
          \begin{tikzpicture}

            \node[main node] (1) {};
            \node[main node] (0) [left=1cm of 1] {};
            \node[main node] (2) [right=1cm of 1] {};
            \node[main node] (3) [right=1cm of 2] {};
            \node[main node] (6) [right=1cm of 3] {};
            \node[main node] (8) [above=1cm of 2] {};
            \node[main node] (9) [above=1cm of 8] {};

            \node (00) [below=0.2cm of 0] {$a$};
            \node (11) [below=0.1cm of 1] {$b$};
            \node (22) [below=0.2cm of 2] {$c$};
            \node (33) [below=0.1cm of 3] {$d$};
            \node (66) [below=0.2cm of 6] {$e$};
            \node (88) [left=0.1cm of 8] {$f$};
            \node (99) [left=0.1cm of 9] {$g$};

            \path[draw]
            (0)--(1)--(2)--(3)--(6)
            (2)--(8)--(9);
          \label{ae6}
          \end{tikzpicture}
        \end{center}
      \end{lemma}
      \begin{proof}
        Consider the elements
        \[
          w_k=(acbfcgdfecdb)^k
        \]
        for $k\in \Z_{\ge 1}$. The heap of $w_k$ is shown below.         
        
        \begin{center}
          \resizebox{0.4\textwidth}{!}
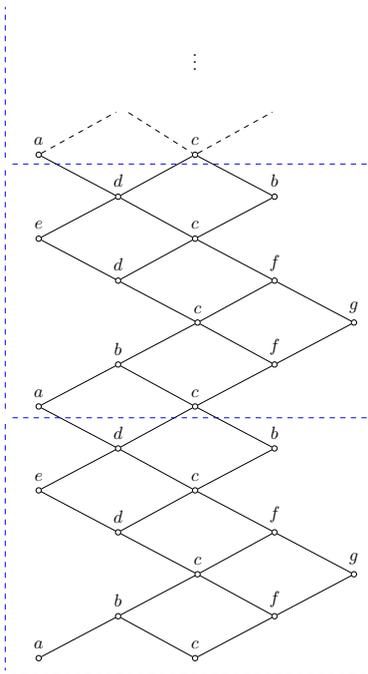
{
            \begin{tikzpicture}
              \node[main node] (1) {};
              \node[main node] (2) [right=3cm of 1] {};
              \node[main node] (3) [above right = 0.75cm and 1.5cm of 1] {};
              \node[main node] (4) [above right = 0.75cm and 1.5cm of 2] {};
              \node[main node] (5) [above right =0.75cm and 1.5cm of 3] {};
              \node[main node] (6) [above right =0.75cm and 1.5cm of 4] {};
              \node[main node] (7) [above left =0.75cm and 1.5cm of 5] {};
              \node[main node] (8) [above left =0.75cm and 1.5cm of 6] {};
              \node[main node] (9) [above left =0.75cm and 1.5cm of 7] {};
              \node[main node] (10) [above left =0.75cm and 1.5cm of 8] {};
              \node[main node] (11) [above right =0.75cm and 1.5cm of 9] {};
              \node[main node] (12) [above right =0.75cm and 1.5cm of 10] {};
              \node[main node] (b1) [above left= 0.75cm and 1.5cm of 11] {};
              \node[main node] (b2) [above left= 0.75cm and 1.5cm of 12] {};
              \node[main node] (b3) [above right = 0.75cm and 1.5cm of b1] {};
              \node[main node] (b4) [above right = 0.75cm and 1.5cm of b2] {};
              \node[main node] (b5) [above right =0.75cm and 1.5cm of b3] {};
              \node[main node] (b6) [above right =0.75cm and 1.5cm of b4] {};
              \node[main node] (b7) [above left =0.75cm and 1.5cm of b5] {};
              \node[main node] (b8) [above left =0.75cm and 1.5cm of b6] {};
              \node[main node] (b9) [above left =0.75cm and 1.5cm of b7] {};
              \node[main node] (b10) [above left =0.75cm and 1.5cm of b8] {};
              \node[main node] (b11) [above right =0.75cm and 1.5cm of b9] {};
              \node[main node] (b12) [above right =0.75cm and 1.5cm of b10] {};
              \node[main node] (bb1) [above left=0.75cm and 1.5cm of b11] {}; 
              \node[main node] (bb2) [above left=0.75cm and 1.5cm of b12] {};
              \node (bb3) [above right=0.75cm and 1.5cm of bb1] {}; 
              \node (bb4) [above right=0.75cm and 1.5cm of bb2] {};

              \node (c1) [above=0.01cm of 1] {$a$};
              \node (c2) [above=0.01cm of 2] {$c$};
              \node (c3) [above=0.01cm of 3] {$b$};
              \node (c4) [above=0.01cm of 4] {$f$};
              \node (c5) [above=0.01cm of 5] {$c$};
              \node (c6) [above=0.01cm of 6] {$g$};
              \node (c7) [above=0.01cm of 7] {$d$};
              \node (c8) [above=0.01cm of 8] {$f$};
              \node (c9) [above=0.01cm of 9] {$e$};
              \node (c10) [above=0.01cm of 10] {$c$};
              \node (c11) [above=0.01cm of 11] {$d$};
              \node (c12) [above=0.01cm of 12] {$b$};
              \node (d1) [above=0.01cm of b1] {$a$};
              \node (d2) [above=0.01cm of b2] {$c$};
              \node (d3) [above=0.01cm of b3] {$b$};
              \node (d4) [above=0.01cm of b4] {$f$};
              \node (d5) [above=0.01cm of b5] {$c$};
              \node (d6) [above=0.01cm of b6] {$g$};
              \node (d7) [above=0.01cm of b7] {$d$};
              \node (d8) [above=0.01cm of b8] {$f$};
              \node (d9) [above=0.01cm of b9] {$e$};
              \node (d10) [above=0.01cm of b10] {$c$};
              \node (d11) [above=0.01cm of b11] {$d$};
              \node (d12) [above=0.01cm of b12] {$b$};
              \node (dd1) [above=0.01cm of bb1] {$a$};
              \node (dd2) [above=0.01cm of bb2] {$c$};

              \path[draw]
              (1)--(3)--(5)--(7)--(9)--(11)--(b1)
              (2)--(4)--(6)--(8)--(10)--(12)--(b2)
              (2)--(3)
              (4)--(5)--(8)
              (7)--(10)--(11)--(b2)
              (b1)--(b3)--(b5)--(b7)--(b9)--(b11)--(bb1)
              (b2)--(b4)--(b6)--(b8)--(b10)--(b12)--(bb2)
              (b2)--(b3)
              (b4)--(b5)--(b8)
              (b7)--(b10)--(b11)--(bb2);

              \path[draw,dashed]
              (bb1)--(bb3)--(bb2)--(bb4);

              \node (r1) [below left = 0.2cm and 0.5cm of 1] {};
              \node (r2) [below right = 0.2cm and 6.5cm of 1] {};
              \node (r3) [above = 4.9cm of r1] {};
              \node (r4) [above = 4.9cm of r2] {};
              \node (r5) [above = 4.8cm of r3] {};
              \node (r6) [above = 4.8cm of r4] {};
              \node (r7) [above = 3cm of r5] {};
              \node (r8) [above = 3cm of r6] {};
              \node (r9) [above = 1.5cm of bb2] {$\vdots$};

              \path[draw,dashed,blue]
              (r1)--(r2)--(r4)--(r3)--(r1)
              (r4)--(r6)--(r5)--(r3)
              (r5)--(r7)
              (r6)--(r8);
            \end{tikzpicture} 
          }
        \end{center}

As in the previous lemmas, we may observe from the above figure that $w_k$ is reduced and fully commutative, and that $n(w_k)=2$,
        for each $k\ge
        1$.  
        Since the subgraph in question is the Coxeter diagram of $\tilde E_6$, an affine Weyl
        group of type $E$, Proposition \ref{a from heap} implies that $\la(w_k)=2$ for all
        $k\ge 1$, therefore $W$ is $\la(2)$-infinite.  
      \end{proof}

      Our next lemma will rely on the following non-trivial result from  \cite{star_reducible}.

      \begin{prop}[{\cite[Lemma 5.5]{star_reducible}}]
        \label{affine F5}
        The Coxeter group with the following Coxeter diagram is star reducible.
        
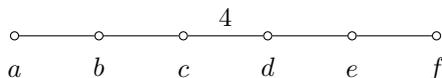
\begin{center}
          \begin{tikzpicture}

            \node[main node] (1) {};
            \node[main node] (0) [left=1cm of 1] {};
            \node[main node] (2) [right=1cm of 1] {};
            \node[main node] (3) [right=1cm of 2] {};
            \node[main node] (6) [right=1cm of 3] {};
            \node[main node] (7) [right=1cm of 6] {};

            \node (00) [below=0.2cm of 0] {$a$};
            \node (11) [below=0.1cm of 1] {$b$};
            \node (22) [below=0.2cm of 2] {$c$};
            \node (33) [below=0.1cm of 3] {$d$};
            \node (66) [below=0.2cm of 6] {$e$};
            \node (77) [below=0.1cm of 7] {$f$};

            \path[draw]
            (0)--(1)
            (6)--(7)
            (1) edge node {} (2)
            (2) edge node [above] {$4$} (3)
            (3) edge node {} (6);
\label{af5}
          \end{tikzpicture}

        \end{center}
      \end{prop}

      \begin{lemma}
        \label{affine F5 lemma}
        If $G$ contains a subgraph of the form shown in Proposition
        \ref{affine F5}, then
        $W$ is $\la(2)$-infinite.  
      \end{lemma}
      \begin{proof}
        Consider the elements
        \[
          w_k=(bdacbdcedfce)^k
        \]
        for $k\in \Z_{\ge 1}$. The heap of $w_k$ is shown below.
        \begin{center}
          \resizebox{0.5\textwidth}{!}
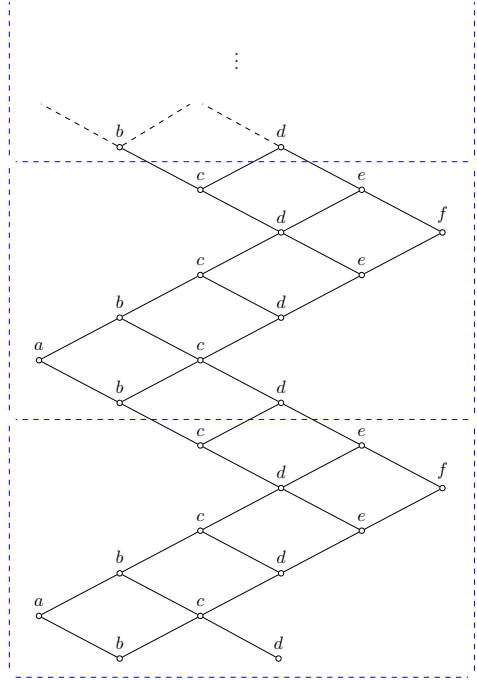
{
            \begin{tikzpicture}
              \node[main node] (1) {};
              \node[main node] (2) [right=3cm of 1] {};
              \node[main node] (3) [above left = 0.75cm and 1.5cm of 1] {};
              \node[main node] (4) [above right = 0.75cm and 1.5cm of 1] {};
              \node[main node] (5) [above right =0.75cm and 1.5cm of 3] {};
              \node[main node] (6) [above right =0.75cm and 1.5cm of 4] {};
              \node[main node] (7) [above right =0.75cm and 1.5cm of 5] {};
              \node[main node] (8) [above right =0.75cm and 1.5cm of 6] {};
              \node[main node] (9) [above right =0.75cm and 1.5cm of 7] {};
              \node[main node] (10) [above right =0.75cm and 1.5cm of 8] {};
              \node[main node] (11) [above left =0.75cm and 1.5cm of 9] {};
              \node[main node] (12) [above left =0.75cm and 1.5cm of 10] {};
              \node[main node] (b1) [above left= 0.75cm and 1.5cm of 11] {};
              \node[main node] (b2) [above left= 0.75cm and 1.5cm of 12] {};
              \node[main node] (b3) [above left = 0.75cm and 1.5cm of b1] {};
              \node[main node] (b4) [above right = 0.75cm and 1.5cm of b1] {};
              \node[main node] (b5) [above right =0.75cm and 1.5cm of b3] {};
              \node[main node] (b6) [above right =0.75cm and 1.5cm of b4] {};
              \node[main node] (b7) [above right =0.75cm and 1.5cm of b5] {};
              \node[main node] (b8) [above right =0.75cm and 1.5cm of b6] {};
              \node[main node] (b9) [above right =0.75cm and 1.5cm of b7] {};
              \node[main node] (b10) [above right =0.75cm and 1.5cm of b8] {};
              \node[main node] (b11) [above left =0.75cm and 1.5cm of b9] {};
              \node[main node] (b12) [above left =0.75cm and 1.5cm of b10] {};
              \node[main node] (bb1) [above left=0.75cm and 1.5cm of b11] {}; 
              \node[main node] (bb2) [above left=0.75cm and 1.5cm of b12] {};
              \node (bb3) [above left=0.75cm and 1.5cm of bb1] {}; 
              \node (bb4) [above left=0.75cm and 1.5cm of bb2] {};

              \node (c1) [above=0.01cm of 1] {$b$};
              \node (c2) [above=0.01cm of 2] {$d$};
              \node (c3) [above=0.01cm of 3] {$a$};
              \node (c4) [above=0.01cm of 4] {$c$};
              \node (c5) [above=0.01cm of 5] {$b$};
              \node (c6) [above=0.01cm of 6] {$d$};
              \node (c7) [above=0.01cm of 7] {$c$};
              \node (c8) [above=0.01cm of 8] {$e$};
              \node (c9) [above=0.01cm of 9] {$d$};
              \node (c10) [above=0.01cm of 10] {$f$};
              \node (c11) [above=0.01cm of 11] {$c$};
              \node (c12) [above=0.01cm of 12] {$e$};
              \node (d1) [above=0.01cm of b1] {$b$};
              \node (d2) [above=0.01cm of b2] {$d$};
              \node (d3) [above=0.01cm of b3] {$a$};
              \node (d4) [above=0.01cm of b4] {$c$};
              \node (d5) [above=0.01cm of b5] {$b$};
              \node (d6) [above=0.01cm of b6] {$d$};
              \node (d7) [above=0.01cm of b7] {$c$};
              \node (d8) [above=0.01cm of b8] {$e$};
              \node (d9) [above=0.01cm of b9] {$d$};
              \node (d10) [above=0.01cm of b10] {$f$};
              \node (d11) [above=0.01cm of b11] {$c$};
              \node (d12) [above=0.01cm of b12] {$e$};
              \node (dd1) [above=0.01cm of bb1] {$b$};
              \node (dd2) [above=0.01cm of bb2] {$d$};

              \path[draw]
              (1)--(4)--(6)--(8)--(10)--(12)--(b2)--(b4)--(b6)--(b8)--(b10)--(b12)
              (1)--(3)--(5)--(7)--(9)--(11)--(b1)--(b3)--(b5)--(b7)--(b9)--(b11)
              (2)--(4)--(5)
              (6)--(7)
              (8)--(9)--(12)
              (11)--(b2)
              (b1)--(b4)--(b5)
              (b6)--(b7)
              (b8)--(b9)--(b12)
              (b12)--(bb2)--(b11)--(bb1);

              \path[draw,dashed]
              (bb1)--(bb3)
              (bb1)--(bb4)
              (bb2)--(bb4);

              \node (r1) [below left = 0.2cm and 2cm of 1] {};
              \node (r2) [below right = 0.2cm and 6.8cm of 1] {};
              \node (r3) [above = 4.8cm of r1] {};
              \node (r4) [above = 4.8cm of r2] {};
              \node (r5) [above = 4.8cm of r3] {};
              \node (r6) [above = 4.8cm of r4] {};
              \node (r7) [above = 3cm of r5] {};
              \node (r8) [above = 3cm of r6] {};
              \node (r9) [above right= 2.2cm and 0.5cm of b11] {$\vdots$};

              \path[draw,dashed,blue]
              (r1)--(r2)--(r4)--(r3)--(r1)
              (r4)--(r6)--(r5)--(r3)
              (r5)--(r7)
              (r6)--(r8);

            \end{tikzpicture}
          } 
        \end{center}

As in the previous lemmas, it is clear from the above figure that
$w_k$ is reduced and fully commutative and that $n(w_k)=2$ for each $k\ge 1$.
Since the subgraph in question corresponds to a star reducible Coxeter group by Proposition
        \ref{affine F5}, it follows from Proposition \ref{star reducible a} that
        $\la(w_k)=2$ for all $k\ge 1$, therefore $W$ is $\la(2)$-infinite. 
 \end{proof}

 \subsection{Lemmas with star operation arguments} 
 \label{star arguments}
 For our second set of lemmas, the proofs that our witnesses have
 $\la$-value 2 will involve the star operations introduced in Section
 \ref{star}. Our main tools will be Corollary \ref{star a}
 and Remark \ref{removal}.

The first lemma in this set deals with the case where $G$ contains a cycle. It
will be used to prove the ``only if'' direction of Theorem \ref{first
theorem}.(1).
     \begin{lemma} 
        \label{cycle witness} Suppose $G$ contains a cycle
      $C=(v_1,v_2,\cdots,v_n,v_1)$ for some $n\ge 3$.  \begin{enumerate} \item  If
            $G$ contains a vertex $v$ that does not not appear in $C$ and is not
          adjacent to all vertices in $C$, then $W$ is $\la(2)$-infinite.  \item If $C$
            contains two vertices that are not adjacent, then $W$ is
            $\la(2)$-infinite.
        \end{enumerate} 
      \end{lemma} 
      \begin{proof}

        (1) Suppose $v$ is not adjacent to $v_j$ for some $1\le j\le n$. Consider the
        elements 
        \[
          x_k=vv_j(v_{j+1}v_{j+2}\cdots v_nv_1\cdots v_{j-1}v_j)^k
        \]
        for  $k\in \Z_{\ge 0}$. For each $k\ge 1$, note that $x_k$ is reduced (and is actually a reduced word of a fully
        commutative element) by Proposition \ref{fc criterion}. Moreover, by
        Remark \ref{removal}, we may
        reduce $x_k$ to $x_{k-1}$ via $n$ lower right star operations, successively
        with respect to the
        pairs
        \[
          \{v_j,v_{j-1}\},\{v_{j-1},v_{j-2}\},\cdots,\{v_{n},v_{n-1}\},\cdots,
        \{v_{j+1},v_{j}\}. \]
        It follows that $\la(x_k)=\la(x_0)$ for all $k\ge 0$. Since
        $\la(x_0)=\la(vv_j)=2$ by Corollary \ref{product a}, it further follows that
        $\la(x_k)=2$ for all
        $k\ge 0$, therefore $W$  is $\la(2)$-infinite.

        (2) Suppose $v_i,v_j$ are not adjacent for some $1\le i, j\le n$. Let
        \[
          y_k=v_iv_j(v_{j+1}v_{j+2}\cdots v_nv_{1}\cdots v_{j-1}v_{j})^k
        \]
        for $k\in \Z_{\ge 0}$. Then by an argument similar to the one in (1), $y_k$ is right star reducible to $y_0$ and
        $\la(y_k)=\la(y_0)=2$ for all $k\ge 0$, therefore $W$ is $\la(2)$-infinite.
      \end{proof}






      \begin{lemma}
        \label{two strong bonds}
        If $G$ contains a subgraph of the form 
        
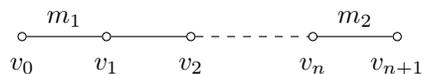
\begin{center}
          \begin{tikzpicture}
            \node[main node] (1) {};
            \node[main node] (2) [right=1cm of 1] {};
            \node[main node] (3) [right=1cm of 2] {};
            \node[main node] (6) [right=1.5cm of 3] {};
            \node[main node] (7) [right=1cm of 6] {};

            \node (11) [below=0.1cm of 1] {$v_0$};
            \node (22) [below=0.1cm of 2] {$v_1$};
            \node (33) [below=0.1cm of 3] {$v_2$};
            \node (66) [below=0.1cm of 6] {$v_n$};
            \node (77) [below=0.1cm of 7] {$v_{n+1}$};

            \path[draw]
            (1) edge node [above] {$m_1$} (2)
            (2) edge node {} (3)
            (6) edge node [above] {$m_2$} (7);     
            \path[draw,dashed]
            (3)--(6);
          \end{tikzpicture}
        \end{center}
        \noindent where $n\ge 1$, $m_1\ge 5, m_2\ge 4$ and all the middle edges have weight 3, then
$W$ is $\la(2)$-infinite.
 
      \end{lemma}
      \begin{proof}
        Consider the elements
        \[
          w_k=v_0(v_{n+1}v_nv_{n-1}\cdots v_1v_0v_1\cdots v_{n})^k
        \]
        for $k\in \Z_{\ge 1}$. The heap of $w_k$ is shown below. Note that by Proposition
        \ref{fc heap criterion}, it is clear from the figure that $w_k$ is reduced for each $k\ge 1$. 

        \begin{center}
          \resizebox{0.5\textwidth}{!}
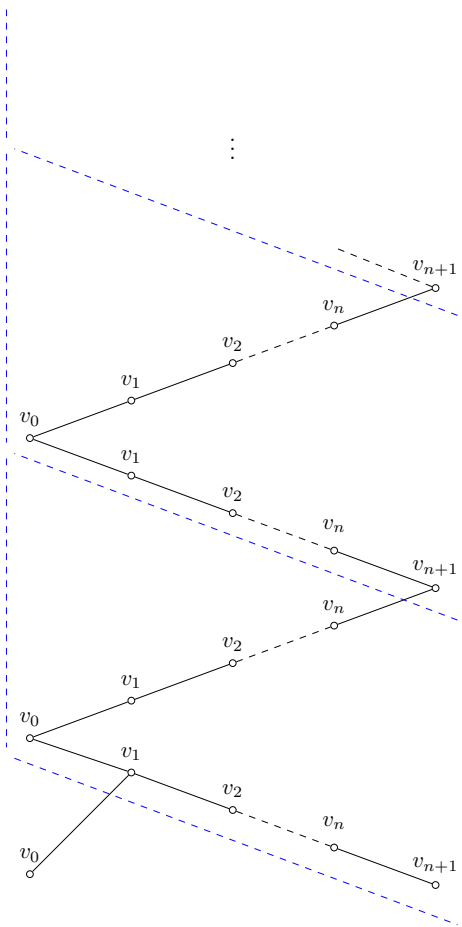
{
            \begin{tikzpicture}
              \node[main node] (1) {};
              \node (v1) [above=0.01cm of 1] {$v_0$};
              \node[main node] (2) [above right=1.5cm and 1.5cm of 1] {};
              \node (v2) [above= 0.01cm of 2] {$v_1$};
              \node[main node] (3) [below right=0.5cm and 1.5cm  of 2] {};
              \node (v3) [above=0.01cm of 3]  {$v_2$};
              \node[main node] (4) [below right=0.5cm and 1.5cm  of 3] {};
              \node (v4) [above=0.1cm of 4] {$v_{n}$};
              \node[main node] (5) [below right=0.5cm and 1.5cm of 4] {};
              \node (v5) [above=0.01cm of 5] {$v_{n+1}$};
              \node[main node] (6) [above=2cm of 1] {};
              \node (v6) [above=0.01cm of 6] {$v_0$};
              \node[main node] (7) [above right=0.5cm and 1.5cm of 6] {};
              \node (v7) [above=0.01cm of 7] {$v_1$};
              \node[main node] (8) [above right=0.5cm and 1.5cm  of 7] {};
              \node (v8) [above=0.01cm of 8] {$v_2$};
              \node[main node] (9) [above right=0.5cm and 1.5cm  of 8] {};
              \node (v9) [above=0.01cm of 9] {$v_{n}$};
              \node[main node] (10) [above right=0.5cm and 1.5cm  of 9] {};
              \node (v10) [above=0.01cm of 10] {$v_{n+1}$};
              \node[main node] (11) [above left=0.5cm and 1.5cm  of 10] {};
              \node (v11) [above=0.1cm of 11] {$v_{n}$};
              \node[main node] (12) [above left=0.5cm and 1.5cm of 11] {};
              \node (v12) [above=0.01cm of 12] {$v_2$};
              \node[main node] (13) [above left=0.5cm and 1.5cm  of 12] {};
              \node (v13) [above=0.01cm of 13] {$v_1$};
              \node[main node] (14) [above left=0.5cm and 1.5cm  of 13] {};
              \node (v14) [above=0.01cm of 14] {$v_{0}$};
              \node[main node] (15) [above right=0.5cm and 1.5cm  of 14] {};
              \node (v15) [above=0.01cm of 15] {$v_{1}$};
              \node[main node] (16) [above right=0.5cm and 1.5cm  of 15] {};
              \node (v16) [above=0.01cm of 16] {$v_2$};
              \node[main node] (17) [above right=0.5cm and 1.5cm  of 16] {};
              \node (v17) [above=0.01cm of 17] {$v_n$};
              \node[main node] (18) [above right=0.5cm and 1.5cm  of 17] {};
              \node (v18) [above=0.01cm of 18] {$v_{n+1}$};
              \node (a19) [above left=0.5cm and 1.5cm  of 18] {};

              \node (19) [below left = 0.1cm and 0.2cm of 6] {};
              \node (20) [below right = 0.5cm and 0.3cm of 5] {};
              \node (21) [above=4.5cm of 20] {};
              \node (22) [above=4.5cm of 19] {};
              \node (23) [above=4.5cm of 21] {};
              \node (24) [above=4.5cm of 22] {};
              \node (25) [above=2cm of 23] {};
              \node (26) [above=2cm of 24] {};
              \node (27) [above=3cm of 16] {$\vdots$};

              \path[draw]
              (1)--(2);
              \path[draw]
              (2)--(3)
              (4)--(5)
              (2)--(6)--(7)--(8)
              (9)--(10)--(11)
              (12)--(13)--(14)--(15)--(16)
              (17)--(18);

              \path[draw,dashed]
              (3)--(4)
              (8)--(9)
              (11)--(12)
              (16)--(17)
              (18)--(a19);

              \path[draw,blue,dashed]
              (19)--(20)--(21)--(22)--(19)
              (21)--(23)--(24)--(22)
              (23)--(25)
              (24)--(26);
            \end{tikzpicture}
          }
        \end{center}

        We may commute
        $v_0$ past the first occurrences of $v_{n+1}, v_n,\cdots, v_2$ to write 
        \[
          w_k=v_{n+1}v_n\cdots v_3v_2\cdot v_0\cdot v_1v_0v_1v_2\cdots v_{n}\cdot
          (v_{n+1}v_{n}\cdots v_1v_0v_1\cdots v_n)^{k-1}.
        \]
        By Remark \ref{removal},
        we may then use suitable lower star operations to remove letters from the left
        and right of $w_k$ to obtain $v_2v_0$. (Note that by the discussion following
        Proposition \ref{fc heap criterion}, these operations can be visualized
        as the successive removal of the vertices $v_{n+1}, v_n, \cdots, v_3$
        in the bottom right of the heap $H(w_k)$ and of the vertices
        $v_n,v_{n-1},\cdots, v_1,v_0,v_1,\cdots, v_1,\cdots$ along the top
        ``zig-zag'' part of $H(w_k)$.) This implies that $\la(w_k)=\la(v_2v_0)=2$ for all $k\ge 1$ by corollaries \ref{star a} and
        \ref{product a}, therefore 
$W$ is $\la(2)$-infinite.
      \end{proof}


      \begin{lemma}
        \label{nonextreme 4s}
        If $G$ contains a subgraph of the form 
        
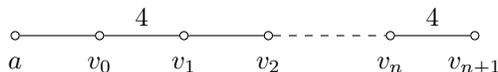
\begin{center}
          \begin{tikzpicture}

            \node[main node] (1) {};
            \node[main node] (2) [right=1cm of 1] {};
            \node[main node] (3) [right=1cm of 2] {};
            \node[main node] (4) [right=1cm of 3] {};
            \node[main node] (6) [right=1.5cm of 4] {};
            \node[main node] (7) [right=1cm of 6] {};

            \node (11) [below=0.1cm of 1] {$a$};
            \node (22) [below=0.1cm of 2] {$v_0$};
            \node (33) [below=0.1cm of 3] {$v_1$};
            \node (44) [below=0.1cm of 4] {$v_2$};
            \node (66) [below=0.1cm of 6] {$v_n$};
            \node (77) [below=0.1cm of 7] {$v_{n+1}$};

            \path[draw]
            (1) edge node {} (2)
            (2) edge node [above] {$4$} (3)
            (3) edge node {} (4)
            (6) edge node [above] {$4$} (7);     
            \path[draw,dashed]
            (4)--(6);
          \end{tikzpicture}
        \end{center}
        \noindent where $n\ge 1$ and all edges between $v_1$ and $v_n$ have
        weight 3, then $W$ is $\la(2)$-infinite.  
           \end{lemma}

      \begin{proof}

        Consider the elements
        \[
          w_k=av_{1}(v_0v_{1}\cdots v_{n}v_{n+1}v_n\cdots v_{2}v_{1})^k
        \]
        for $k\in \Z_{\ge 0}$. The heap of $w_k$ is shown below. Observe that $w_k$ is reduced by Proposition
        \ref{fc heap criterion}. 

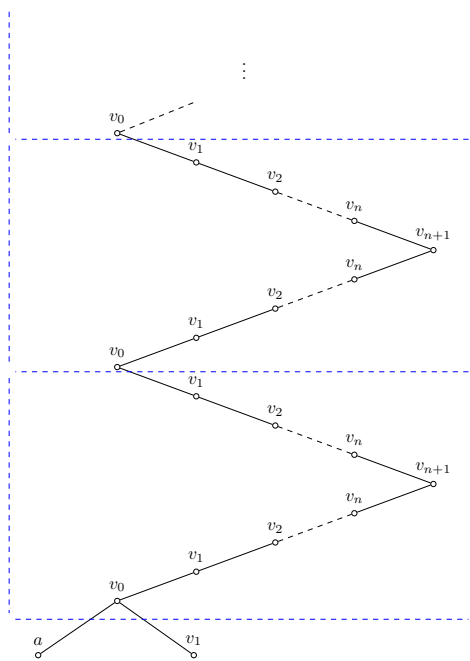
        \begin{center}
          \resizebox{0.5\textwidth}{!}{
            \begin{tikzpicture}
              \node[main node] (1) {};
              \node[main node] (2) [right=3cm of 1] {};
              \node[main node] (3) [above right=1cm and 1.5cm of 1] {};
              \node[main node] (4) [above right=0.5cm and 1.5cm of 3] {};
              \node[main node] (5) [above right=0.5cm and 1.5cm of 4] {};
              \node[main node] (6) [above right=0.5cm and 1.5cm of 5] {};
              \node[main node] (7) [above right=0.5cm and 1.5cm of 6] {};
              \node[main node] (8) [above left=0.5cm and 1.5cm of 7] {};
              \node[main node] (9) [above left=0.5cm and 1.5cm of 8] {};
              \node[main node] (10) [above left=0.5cm and 1.5cm of 9] {};
              \node[main node] (b3) [above left=0.5cm and 1.5cm of 10] {};
              \node[main node] (b4) [above right=0.5cm and 1.5cm of b3] {};
              \node[main node] (b5) [above right=0.5cm and 1.5cm of b4] {};
              \node[main node] (b6) [above right=0.5cm and 1.5cm of b5] {};
              \node[main node] (b7) [above right=0.5cm and 1.5cm of b6] {};
              \node[main node] (b8) [above left=0.5cm and 1.5cm of b7] {};
              \node[main node] (b9) [above left=0.5cm and 1.5cm of b8] {};
              \node[main node] (b10) [above left=0.5cm and 1.5cm of b9] {};
              \node[main node] (bb3) [above left=0.5cm and 1.5cm of b10] {};
              \node (bb4) [above right=0.5cm and 1.5cm of bb3] {};

              \node [above=0.01cm of 1] {$a$}; 
              \node [above=0.01cm of 2] {$v_1$}; 
              \node [above=0.01cm of 3] {$v_0$}; 
              \node [above=0.01cm of 4] {$v_1$}; 
              \node [above=0.01cm of 5] {$v_2$}; 
              \node [above=0.01cm of 6] {$v_n$}; 
              \node [above=0.01cm of 7] {$v_{n+1}$}; 
              \node [above=0.01cm of 8] {$v_n$}; 
              \node [above=0.01cm of 9] {$v_2$}; 
              \node [above=0.01cm of 10] {$v_1$}; 
              \node [above=0.01cm of b3] {$v_0$}; 
              \node [above=0.01cm of b4] {$v_1$}; 
              \node [above=0.01cm of b5] {$v_2$}; 
              \node [above=0.01cm of b6] {$v_n$}; 
              \node [above=0.01cm of b7] {$v_{n+1}$}; 
              \node [above=0.01cm of b8] {$v_n$}; 
              \node [above=0.01cm of b9] {$v_2$}; 
              \node [above=0.01cm of b10] {$v_1$}; 
              \node [above=0.01cm of bb3] {$v_0$}; 

              \path[draw]
              (1)--(3)
              (2)--(3)--(4)--(5)
              (6)--(7)--(8)
              (9)--(10)--(b3)--(b4)--(b5)
              (b6)--(b7)--(b8)
              (b9)--(b10)--(bb3);

              \path[draw,dashed]
              (5)--(6)
              (8)--(9)
              (b5)--(b6)
              (b8)--(b9)
              (bb3)--(bb4);

              \node (r1) [below left = 0.2cm and 2cm of 3] {};
              \node (r2) [below right = 0.2cm and 7cm of 3] {};
              \node (r3) [above = 4.7cm of r1] {};
              \node (r4) [above = 4.7cm of r2] {};
              \node (r5) [above = 4.4cm of r3] {};
              \node (r6) [above = 4.4cm of r4] {};
              \node (r7) [above = 2.5cm of r5] {};
              \node (r8) [above = 2.5cm of r6] {};
              \node (r9) [above right=1.5cm and 0.75cm of b10] {$\vdots$};

              \path[draw,dashed,blue]
              (r1)--(r2)--(r4)--(r3)--(r1)
              (r4)--(r6)--(r5)--(r3)
              (r5)--(r7)
              (r6)--(r8);
            \end{tikzpicture}
          }
        \end{center}

        By Remark \ref{removal}, we may easily obtain $w_k$ from
        $w_{k+1}$        
        via suitable lower right star operations for any $k\ge 0$. (As in the
        proof of the previous lemma, these operations can be easily visualized
      in terms of the heap $H(w_{k+1})$, this time as the successive removal of the
      suitable vertices from the top, wedge-shaped part of the heap.)
        Corollaries \ref{star a} and \ref{product a} then imply that
        $\la(w_k)=\la(w_0)=\la(av_1)=2$ for all $k\ge 0$, therefore $W$ is
        $\la(2)$-infinite.  
      \end{proof}


      \subsection{Lemmas with $\mu$-coefficient computations} 
      \label{mu arguments} 
      For our third
      set of lemmas, the proofs will all
      involve showing $x\prec_R y$ for some elements $x,y$ by using Proposition
      \ref{alternative prec}. The proofs will be more technical than those for
      the previous lemmas, as we will frequently need to use
      propositions \ref{star and mu} and \ref{star relation} to deduce
      $\mu$-values. 

      \begin{lemma}
        \label{strength 6}
        If $G$ contains a subgraph of the form

        
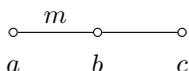
\begin{center}
          \begin{tikzpicture}
            \node[main node] (1) {};
            \node[main node] (2) [right = 1cm of 1] {};
            \node[main node] (3) [right = 1cm of 2] {};

            \node (11) [below = 0.18cm of 1] {$a$};
            \node (22) [below = 0.1cm of 2] {$b$};
            \node (33) [below = 0.18cm of 3] {$c$};

            \path[draw]
            (1) edge node [above] {$m$} (2)
            (2) edge node {} (3);
          \end{tikzpicture}
        \end{center}

        \noindent where $m\ge 6$, then 
$W$ is $\la(2)$-infinite.
      \end{lemma}

      \begin{proof}
        Let $w_k=(cabab)^k$ for each $k\in \Z_{\ge 0}$.
        The heap of $w_k$ is shown below. Observe that for each $k\ge 0$, $w_k$, $w_ka$
        and $w_kca$  are all reduced
        by Proposition \ref{fc heap criterion}. 

        \begin{center}
          \resizebox{0.4\textwidth}{!}{
            \begin{tikzpicture}
              \node[main node] (4) {};
              \node (44) [above = 0.01cm of 4] {$b$};
              \node[main node] (0) [below right= 0.75cm and 1.5cm of 4] {};
              \node (00) [above = 0.01cm of 0] {$c$};
              \node[main node] (1) [below left= 0.75cm and 1.5cm of 4] {};
              \node (11) [above = 0.01cm of 1] {$a$};
              \node[main node] (5) [above left = 0.75cm and 1.5cm of 4] {};
              \node (55) [above = 0.01cm of 5] {$a$};
              \node[main node] (6) [above right =0.75cm and 1.5cm of 5] {};
              \node (66) [above = 0.01cm of 6] {$b$};
              \node[main node] (7) [above left =0.75cm and 1.5cm of 6] {};
              \node (77) [above=0.01cm of 7] {$a$};
              \node[main node] (8) [above right =0.75cm and 1.5cm of 6] {};
              \node (88) [above=0.01cm of 8] {$c$};
              \node[main node] (a4) [above right =0.75cm and 1.5cm of 7] {};
              \node (a44) [above=0.01cm of a4] {$b$};
              \node[main node] (a5) [above left = 0.75cm and 1.5cm of a4] {};
              \node (a55) [above = 0.01cm of a5] {$a$};
              \node[main node] (a6) [above right =0.75cm and 1.5cm of a5] {};
              \node (a66) [above = 0.01cm of a6] {$b$};
              \node[main node] (a7) [above left =0.75cm and 1.5cm of a6] {};
              \node (a77) [above=0.01cm of a7] {$a$};
              \node[main node] (a8) [above right =0.75cm and 1.5cm of a6] {};
              \node (a88) [above=0.01cm of a8] {$c$};
              \node (b4) [above right =0.75cm and 1.5cm of a7] {};

              \node (r1) [below left = 0.2cm and 0.5cm of 1] {};
              \node (r2) [below right = 0.2cm and 0.5cm of 0] {};
              \node (r3) [above = 3.2cm of r1] {};
              \node (r4) [above = 3.2cm of r2] {};
              \node (r5) [above = 3.1cm of r3] {};
              \node (r6) [above = 3.1cm of r4] {};
              \node (r7) [above = 2.1cm of r5] {};
              \node (r8) [above = 2.1cm of r6] {};
              \node (r9) [above = 2cm of a6] {$\vdots$};

              \path[draw] 
              (1)--(4)
              (0)--(4)--(5)--(6)--(7)--(a4)--(8)--(6)
              (a4)--(a5)--(a6)--(a7)
              (a8)--(a6);

              \path[draw,dashed]
              (a7)--(b4)--(a8);

              \path[draw,dashed,blue]
              (r1)--(r2)--(r4)--(r3)--(r1)
              (r3)--(r4)--(r6)--(r5)--(r3)
              (r5)--(r7)
              (r6)--(r8);
            \end{tikzpicture}
          }
        \end{center}

        We shall prove that
        \begin{equation}
          \label{eq:X chain}
          w_kca\le_R w_{k+2}a\le_R w_{k+2}\le_R w_{k+1}\le_R w_kca
        \end{equation}
        for all $k\ge 0$. This implies that $w_k\sim_R ca$ and hence
        $\la(w_k)=\la(ca)=2$ for all $k\ge 1$ by Proposition \ref{constant}
        and Corollary \ref{product a}. It then follows that
$W$ is $\la(2)$-infinite.

        Now let $k\ge 0$ be fixed. To prove \eqref{eq:X chain}, first note that
        $w_{k+2}a\le_R
        w_{k+2}\le_R w_{k+1}\le_R
        w_kca$ by Corollary \ref{weak order}, therefore it suffices to show
        that $w_kca\le_R
        w_{k+2}a$. Let $x=w_kca$ and $y=w_{k+2}a$. We will show that in fact
        $x\prec_R y$. Since
        $x< y$ and $c\in \mathcal{R}(x)\setminus \mathcal{R}(y)$, it further suffices
        to show that $\mu(x,y)\neq 0$ by Proposition \ref{alternative prec}. We do so below.

        Consider the coset decompositions of $x$ and $y$ with respect to $I=\{a,b\}$,
        where
        \begin{IEEEeqnarray*} {RRRLLLLLL}
          x&=& \cdots cababca =  x^I\cdot {x}_I &\quad & \text{with} & \quad
          x^I= w_kc,
          &x_I= a,\\
          y&=& \cdots cababcababa = y^I\cdot y_I & \quad &\text{with} &\quad
          y^I=w_{k+1}c, & y_I= ababa.
        \end{IEEEeqnarray*}
        For any integer $0\le i\le m$, let $\alpha_i$ be the word $ab\cdots$ that
        alternates in $a$ and $b$, starts in $a$, and has length $i$, then let
        $x_i=x^I\cdot \alpha_i$ and $y_i=y^I\cdot \alpha_i$. Set
        \[
          [i,j]=\mu(x_i, y_j)
        \]
        for all $0\le i,j\le m$. Then
        \begin{enumerate}
          \item $[3,1]=0$ by Proposition \ref{extremal}, since $l(x_3)<l(y_1)-1$ and
            $c\in \mathcal{R}(y_1)\setminus\mathcal{R}(x_3)$;
          \item $[2,2]=[1,1]$ by Proposition \ref{star and mu}, since $x_1=x_2*,
            y_1=y_2*$ with respect to the pair $I=\{b,c\}$ and $x_2^{-1}y_2=abcab\notin
            W_{I}$;
          \item $[2,2]=[1,3]+[1,1]$ by Proposition \ref{star relation}, hence
            $[1,3]=0$ by (2);
          \item $[1,3]+[3,3]=[2,4]+[2,2]$, hence $[3,3]=[2,4]+[2,2]$ by (3);
          \item $[4,2]+[2,2]=[3,3]+[3,1]$, hence $[3,3]=[4,2]+[2,2]$ by (1);
          \item $[2,4]=[4,2]$ by (4) and (5);
          \item $[2,4]=[1,5]+[1,3]$ by Proposition \ref{star relation}, hence
            $[1,5]=[2,4]$ by (3);
          \item
            $[4,2]=\mu(w_{k+1}c,w_{k+1}ca)=1$,
            where the second equation follows from Corollary \ref{difference 1}
            and
            the first equation holds by Proposition \ref{star and mu} because with respect to $I=\{b,c\}$, 
            $(w_{k+1}c)*=x_4$, $(w_{k+1}ca)*=y_2$ and
            $(w_{k+1}c)^{-1}(w_{k+1}ca)=a\notin W_I$;     
          \item $\mu(x,y)=[1,5]=[2,4]=[4,2]=1$ by (7), (6), and (8).
        \end{enumerate}
        We have now shown that $\mu(x,y)\neq 0$, and our proof is complete.
      \end{proof}

      \begin{lemma}
        \label{middle 5}
        If $G$ contains a subgraph of the following form, then $W$ is $\la(2)$-infinite. 
 
        
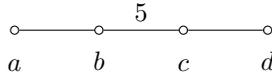
\begin{center}
          \begin{tikzpicture}
            \node[main node] (1) {};
            \node[main node] (2) [right=1cm of 1] {};
            \node[main node] (3) [right=1cm of 2] {};
            \node[main node] (6) [right=1cm of 3] {};

            \node (11) [below=0.2cm of 1] {$a$};
            \node (22) [below=0.1cm of 2] {$b$};
            \node (33) [below=0.2cm of 3] {$c$};
            \node (66) [below=0.1cm of 6] {$d$};

            \path[draw]
            (1) edge node {} (2)
            (2) edge node [above] {$5$} (3)
            (3) edge node {} (6);
          \end{tikzpicture}
        \end{center}
      \end{lemma}

      \begin{proof}
        Let $X=acbc$ and $Y=bdcb$. For each $k\in \Z_{\ge 0}$, let $w_k=XY\cdots$ be
        the string that starts in
        $X$ and contains
        $k$ alternating occurrences of $X$ and $Y$. The heap of $w_k$ is shown
        below, where the dashed rectangles alternately correspond
        to the
        expression $X$ and $Y$ and appear a total of $k$ times. Observe that
        $w_k$ is reduced for each $k$ by Proposition \ref{fc heap criterion}.
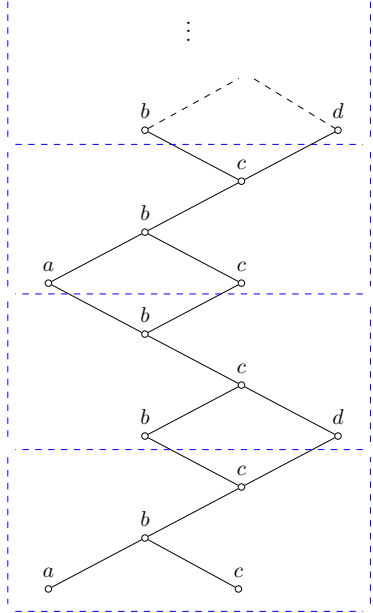
        \begin{center}
          \resizebox{0.4\textwidth}{!}{
            \begin{tikzpicture}
              \node[main node] (3) {};
              \node[main node] (4) [right = 3cm of 3] {};
              \node[main node] (5) [above right=0.75cm and 1.5cm of 3] {};
              \node[main node] (6) [above right=0.75cm and 1.5cm of 5] {};
              \node[main node] (7) [above left=0.75cm and 1.5cm of 6] {};
              \node[main node] (8) [above right=0.75cm and 1.5cm of 6] {};
              \node[main node] (9) [above right=0.75cm and 1.5cm of 7] {};
              \node[main node] (10) [above left=0.75cm and 1.5cm of 9] {};
              \node[main node] (11) [above left=0.75cm and 1.5cm of 10] {};
              \node[main node] (12) [above right=0.75cm and 1.5cm of 10] {};
              \node[main node] (13) [above right=0.75cm and 1.5cm of 11] {};
              \node[main node] (14) [above right=0.75cm and 1.5cm of 13] {};
              \node[main node] (15) [above left=0.75cm and 1.5cm of 14] {};
              \node[main node] (16) [above right=0.75cm and 1.5cm of 14] {};
              \node (17) [above right=0.75cm and 1.5cm of 15] {};
              \node (18) [above left=0.75cm and 1.5cm of 17] {};

              \node (33) [above=0.01cm of 3] {$a$};
              \node (44) [above=0.01cm of 4] {$c$};
              \node (55) [above=0.01cm of 5] {$b$};
              \node (66) [above=0.01cm of 6] {$c$};
              \node (77) [above=0.01cm of 7] {$b$};
              \node (88) [above=0.01cm of 8] {$d$};
              \node (99) [above=0.01cm of 9] {$c$};
              \node (1010) [above=0.01cm of 10] {$b$};
              \node (1111) [above=0.01cm of 11] {$a$};
              \node (1212) [above=0.01cm of 12] {$c$};
              \node (1313) [above=0.01cm of 13] {$b$};
              \node (1414) [above=0.01cm of 14] {$c$};
              \node (1515) [above=0.01cm of 15] {$b$};
              \node (1616) [above=0.01cm of 16] {$d$};

              \path[draw]
              (3)--(5)--(6)--(8)--(9)--(10)--(11)--(13)--(14)--(16)      
              (4)--(5)
              (6)--(7)--(9)
              (10)--(12)--(13)
              (14)--(15);

              \path[draw,dashed]
              (15)--(17)--(16);

              \node (r1) [below left=0.2cm and 0.5cm of 3] {};
              \node (r2) [below right=0.2cm and 2cm of 4] {};
              \node (r3) [above = 2.4cm of r1] {};
              \node (r4) [above = 2.4cm of r2] {};
              \node (r5) [above = 2.3cm of r3] {};
              \node (r6) [above = 2.3cm of r4] {};
              \node (r7) [above = 2.2cm of r5] {};
              \node (r8) [above = 2.2cm of r6] {};
              \node (r9) [above = 2.3cm of r7] {};
              \node (r10) [above = 2.3cm of r8] {};
              \node (r11) [above right=1.3cm and 0.5cm of 15] {$\vdots$};

              \path[draw,dashed,blue]
              (r1)--(r2)--(r4)--(r3)--(r1)
              (r4)--(r6)--(r5)--(r3)
              (r5)--(r7)--(r8)--(r6)
              (r7)--(r9)
              (r8)--(r10);

            \end{tikzpicture}
          }
        \end{center}

        We shall prove that
        \begin{equation}
          \label{eq:even}
          w_kac\le_R w_{k+2}c\le_R w_{k+2}\le_R w_{k+1} \le_R w_kac
        \end{equation}
        for all even integers $k\ge 0$ and that
        \begin{equation}
          \label{eq:odd}
          w_kbd\le_R w_{k+2}b\le_R w_{k+2}\le_R w_{k+1}\le_R w_kbd
        \end{equation}
        for all odd integers $k\ge 1$. It then follows that $w_k\sim_R ac$ and hence
        $\la(w_k)=\la(ac)=2$ for all $k\ge 1$, therefore $W$ is $\la(2)$-infinite.

        To prove \eqref{eq:even}, let $k\ge 0$ be an even integer. Note that
        $w_{k+2}c\le_R w_{k+2}\le_R w_{k+1} \le_R w_kac$ follows from Corollary
        \ref{weak order}, therefore it suffices to show that $w_kac\le_R
        w_{k+2}c$. Let $x=w_kac$ and $y=w_{k+2}c$. We will show  in fact
        $x\prec_R y$. Since $x<y$ and $a\in
        \mathcal{R}(x)\setminus\mathcal{R}(y)$, it
        further suffices to show that $\mu(x,y)\neq 0$ by Proposition
        \ref{alternative prec}.

        To compute $\mu(x,y)$, we consider the coset decompositions of $x$ and
        $y$ with respect to $I=\{b,c\}$, where 
        \begin{IEEEeqnarray*} {RRRLLLLLL}
          x&=& \cdots bdcbac =  x^I\cdot {x}_I &\quad & \text{with} & \quad
          x^I= w_ka,
          & x_I= c,\\
          y&=& \cdots  acbcdbcbc= y^I\cdot y_I & \quad &\text{with} &\quad
          y^I=w_{k+1}d, &{}^I
          y= bcbc.
        \end{IEEEeqnarray*}
        For any integers $0\le i,j \le 4$, let $p_i$ be the word $cb\cdots$ that
        alternates in $b$ and $c$, starts in $c$, and has length $i$, and similarly let
        $q_j$ be the alternating word $bc\cdots$ of length $j$. Let
        $x_i=x^I\cdot p_i$ and $y_j=y^I\cdot q_j$, and set
        \[
          [i,j]=\mu(x_i, y_j)
        \]
        for all $0\le i,j\le 4$. We have
        \begin{eqnarray*}
          [1,4]&=& -[3,4]+[2,3]\\
          &=& -[3,4]+(-[4,3]+[3,2]+[3,4])\\
          &=& -[4,3]+[3,2]\\
          &=& -[4,3]+([4,1]+[4,3])\\
          &=& [4,1],
        \end{eqnarray*}
        where the first, second, and fourth equality follow from applications of Part (1) of
        Proposition \ref{star relation} with $(x_2,y_4), (x_3,y_3)$ and $(x_4,y_2)$ in
        place of the pair $(x,y)$, respectively. Now, 
        \[
          [4,1]=\mu(x_4,y_1)=\mu(w_kacbcb,w_kacbcdb)=1,
        \]
        where the last equality follows from Corollary \ref{difference 1},
        therefore $\mu(x,y)=[1,4]=[4,1]\neq 0$, and we have proved \eqref{eq:even}. 

        The proof of \eqref{eq:odd} is similar to that of $\eqref{eq:even}$,
        thanks to the symmetry $a\leftrightarrow d, b\leftrightarrow c$ in
        the subgraph in question. We have now completed our proof. 
      \end{proof}

      \begin{lemma}
        \label{small 5}
        If $G$ contains a subgraph of the following form, then $W$ is $\la(2)$-infinite. 

        
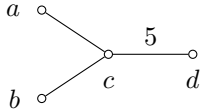
\begin{center}
          \begin{tikzpicture}
            \node[main node] (1) {};
            \node[main node] (8) [above left = 0.5cm and 0.8cm of 1] {};
            \node[main node] (9) [below left = 0.5cm and 0.8cm of 1] {};
            \node[main node] (2) [right=1cm of 1] {};

            \node (88) [left=0.1cm of 8] {$a$};
            \node (99) [left=0.1cm of 9] {$b$};
            \node (11) [below=0.11cm of 1] {$c$};
            \node (22) [below=0.01cm of 2] {$d$};

            \path[draw]
            (8) edge node {} (1)
            (9) edge node {} (1)
            (1) edge node [above] {$5$} (2);
          \end{tikzpicture}
        \end{center}
      \end{lemma}

      \begin{proof}
        We present a proof similar to that of Lemma \ref{middle 5}. Let
        $X=abcdc$ and $Y=bdcdc$. For each $k\in \Z_{\ge 0}$, let
        $w_k=XY\cdots$ be the alternating string starting with $X$ and
        containing $k$ total occurrences of $X$ and $Y$. The heap of $w_k$ is shown
        below, from which we observe that $w_k$ is reduced for each $k$. 

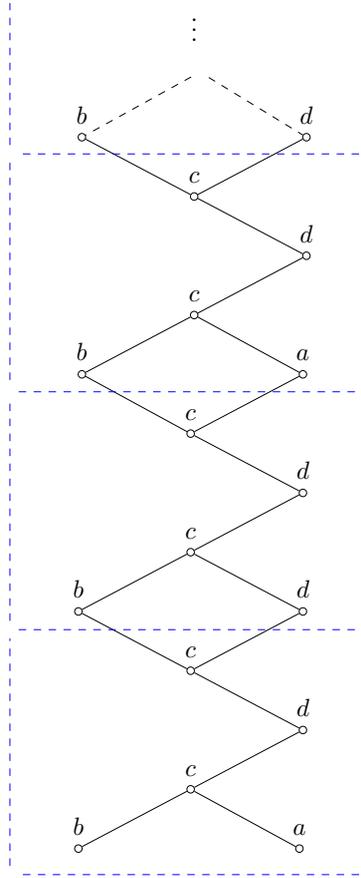
  \begin{center}
          \resizebox{0.4\textwidth}{!}{
            \begin{tikzpicture}

              \node[main node] (1) {};
              \node[main node] (2) [left = 3cm of 1] {};
              \node[main node] (3) [above right=0.75cm and 1.5cm of 2] {};
              \node[main node] (4) [above right=0.75cm and 1.5cm of 3] {};
              \node[main node] (5) [above left=0.75cm and 1.5cm of 4] {};
              \node[main node] (6) [above left=0.75cm and 1.5cm of 5] {};
              \node[main node] (7) [above right=0.75cm and 1.5cm of 5] {};
              \node[main node] (8) [above right=0.75cm and 1.5cm of 6] {};
              \node[main node] (9) [above right=0.75cm and 1.5cm of 8] {};
              \node[main node] (10) [above left=0.75cm and 1.5cm of 9] {};
              \node[main node] (b1) [above right=0.75cm and 1.5cm of 10] {};
              \node[main node] (b2) [left = 3cm of b1] {};
              \node[main node] (b3) [above right=0.75cm and 1.5cm of b2] {};
              \node[main node] (b4) [above right=0.75cm and 1.5cm of b3] {};
              \node[main node] (b5) [above left=0.75cm and 1.5cm of b4] {};
              \node[main node] (b6) [above left=0.75cm and 1.5cm of b5] {};
              \node[main node] (b7) [above right=0.75cm and 1.5cm of b5] {};
              \node (b8) [above right=0.75cm and 1.5cm of b6] {};

              \node (11) [above=0.01cm of 1] {$a$};
              \node (22) [above=0.01cm of 2] {$b$};
              \node (33) [above=0.01cm of 3] {$c$};
              \node (44) [above=0.01cm of 4] {$d$};
              \node (55) [above=0.01cm of 5] {$c$};
              \node (66) [above=0.01cm of 6] {$b$};
              \node (77) [above=0.01cm of 7] {$d$};
              \node (88) [above=0.01cm of 8] {$c$};
              \node (99) [above=0.01cm of 9] {$d$};
              \node (1010) [above=0.01cm of 10] {$c$};
              \node (b11) [above=0.01cm of b1] {$a$};
              \node (b22) [above=0.01cm of b2] {$b$};
              \node (b33) [above=0.01cm of b3] {$c$};
              \node (b44) [above=0.01cm of b4] {$d$};
              \node (b55) [above=0.01cm of b5] {$c$};
              \node (b66) [above=0.01cm of b6] {$b$};
              \node (b77) [above=0.01cm of b7] {$d$};

              \path[draw]
              (2)--(3)--(4)--(5)--(7)--(8)--(9)--(10)--(b2)
              (1)--(3)
              (5)--(6)--(8)
              (10)--(b1)--(b3)
              (b2)--(b3)--(b4)--(b5)
              (b6)--(b5)--(b7);

              \path[draw,dashed]
              (b8)--(b6)
              (b8)--(b7);

              \node (r1) [below right=0.2cm and 0.8cm of 1] {};
              \node (r2) [below left=0.2cm and 0.8cm of 2] {};
              \node (r3) [above = 3.2cm of r1] {};
              \node (r4) [above = 3.2cm of r2] {};
              \node (r5) [above = 3.1cm of r3] {};
              \node (r6) [above = 3.1cm of r4] {};
              \node (r7) [above = 3.1cm of r5] {};
              \node (r8) [above = 3.1cm of r6] {};
              \node (r9) [above = 2cm of r7] {};
              \node (r10) [above = 2cm of r8] {};
              \node (r11) [above =2cm of b5] {$\vdots$};

              \path[draw,dashed,blue]
              (r1)--(r2)--(r4)--(r3)--(r1)
              (r4)--(r6)--(r5)--(r3)
              (r5)--(r7)--(r8)--(r6)
              (r7)--(r9)
              (r8)--(r10);

            \end{tikzpicture}
          }
        \end{center}

        We shall prove that 
        \begin{equation}
          w_{k}ab\le_R w_{k+2}a\le_R w_{k+2}\le_R {w_{k+1}} \le_R w_kab
          \label{eq:ab}
        \end{equation}
        for all even integers $k\ge 0$ and that
        \begin{equation}
          w_kbd\le_R w_{k+2}d\le_R w_{k+2}\le_R w_{k+1}\le_R w_kbd
          \label{eq:bd}
        \end{equation}
        for all odd integers $k\ge 1$. It then follows that $w_k\sim_R ab$ and
        hence $\la(w_k)=\la(ab)=2$ for all $k\ge 1$, therefore $W$ is
        $\la(2)$-infinite.

        We first prove \eqref{eq:ab}. As in Lemma \ref{middle 5}, this
        is easily reduced to showing $w_kab\le_R w_{k+2}a$ and then to
        showing $\mu(w_kab,w_{k+2}a)\neq 0$. By Proposition \ref{star and mu},
        we have $\mu(w_kab,w_{k+2}a)=\mu(w_kabc,w_{k+2})$ by considerations
        with respect to the pair $\{a,c\}$, therefore it further suffices to
        show that $\mu(x,y)\neq 0$ for $x=w_kabc$ and $y=w_{k+2}$. As before,
        we do so by considering the coset decomposition of $x$ and $y$ with respect to
        $I=\{c,d\}$, where 

        \begin{IEEEeqnarray*} {RRRLLLLLL}
          x&=& \cdots abc =  x^I\cdot {x}_I &\quad & \text{with} & \quad
          x^I= w_kab,
          & x_I= c,\\
          y&=& \cdots  bdcdc= y^I\cdot y_I & \quad &\text{with} &\quad
          y^I=w_{k+1}b, &{}^I
          y= dcdc.
        \end{IEEEeqnarray*}
        For any integers $0\le i,j \le 4$, let $p_i$ be the word
        $cd\cdots$ that
        alternates in $c$ and $d$, starts in $c$, and has length $i$, and similarly let
        $q_j$ be the alternating word $dc\cdots$ of length $j$. Let
        $x_i=x^I\cdot p_i$ and $y_j=y^I\cdot q_j$, and set
        \[
          [i,j]=\mu(x_i, y_j)
        \]
        for all $0\le i,j\le 4$. By the same calculations as in Lemma
        \ref{middle 5}, we have
        \[
          \mu(x,y)=[1,4]=[4,1]=\mu(w_kabcdcd,w_{k}abcdcbd)=1,
        \]
        by Corollary \ref{difference 1}, which proves \eqref{eq:ab}.

        We now prove \eqref{eq:bd}. As usual, it suffices to
        show that $\mu(w_kbd,w_{k+2}d)\neq 0$. By setting
        $x=w_kbd, y=w_{k+2}d$, considering their coset decompositions with
        respect to $I=\{c,d\}$, and defining $[i,j]$ for $0\le i,j\le 4$ in the usual way, we may again conclude that
        \[
          \mu(x,y)=[1,4]=[4,1]=\mu(w_kbdcdc,w_{k+1}abc)=\mu(w_{k+1},w_{k+1}abc).
        \]
        Proposition \ref{star and mu}, applied with respect to the
        pair $\{a,c\}$, then implies that
        \[
          \mu(x,y)=\mu(w_{k+1},w_{k+1}abc)=\mu(w_{k+1}a,w_{k+1}ab)=1\neq 0,
        \]
        where the last equality follows from Corollary \ref{difference 1}.
  Our proof is now complete.
      \end{proof}

      \begin{lemma}
        \label{large 5}
        If $G$ contains a subgraph of the form

        
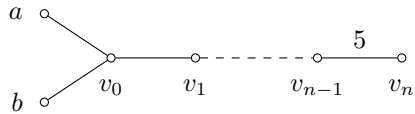
\begin{center}
          \begin{tikzpicture}
            \node[main node] (1) {};
            \node[main node] (8) [above left = 0.5cm and 0.8cm of 1] {};
            \node[main node] (9) [below left = 0.5cm and 0.8cm of 1] {};
            \node[main node] (2) [right=1cm of 1] {};
            \node[main node] (5) [right=1.5cm of 2] {};
            \node[main node] (6) [right=1cm of 5] {};

            \node (88) [left=0.1cm of 8] {$a$};
            \node (99) [left=0.1cm of 9] {$b$};
            \node (11) [below=0.1cm of 1] {$v_0$};
            \node (22) [below=0.1cm of 2] {$v_1$};
            \node (55) [below=0.1cm of 5] {$v_{n-1}$};
            \node (66) [below=0.1cm of 6] {$v_n$};

            \path[draw]
            (8) edge node {} (1)
            (9) edge node {} (1)
            (1) edge node {} (2)
            (5) edge node [above] {$5$} (6);
            \path[draw,dashed]
            (2)--(5);
          \end{tikzpicture}
        \end{center}

        \noindent where $n\ge 2$ and all edges other than $\{v_{n-1},v_n\}$
        have weight 3, then $W$ is $\la(2)$-infinite. 
      \end{lemma}
      \begin{proof}
        We present a proof similar to that of Lemma \ref{small 5}. Let
  \[
 X=abv_0v_1\cdots
  v_{n-1}v_nv_{n-1},\quad Y= v_{n-2}v_nv_{n-1}v_{n}v_{n-1}\cdots v_1v_0,
\]
and define $w_k=XYX\cdots$ 
for each $k\in \Z_{\ge 0}$ as in the previous lemma. The heap of $w_k$ is shown below, from which we
observe that $w_k$ is reduced for each
  $k$.

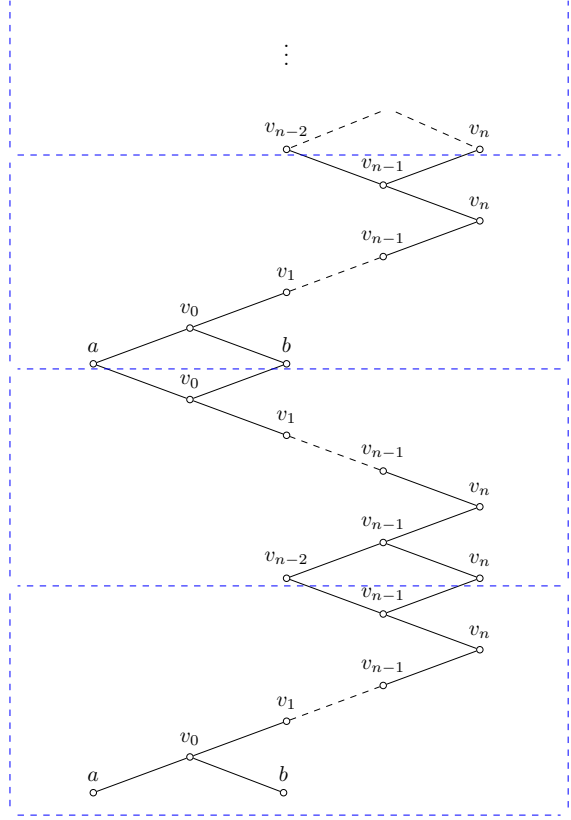
  \begin{center}
          \resizebox{0.6\textwidth}{!}{
            \begin{tikzpicture}
              \node[main node] (8) {};
              \node[main node] (9) [right=3cm of 8] {};
              \node[main node] (10) [above right=0.5cm and 1.5cm of 8] {};
              \node[main node] (11) [above right=0.5cm and 1.5cm of 10] {};
              \node[main node] (12) [above right=0.5cm and 1.5cm of 11] {};
              \node[main node] (13) [above right=0.5cm and 1.5cm of 12] {};
              \node[main node] (14) [above left=0.5cm and 1.5cm of 13] {};
              \node[main node] (15) [above left=0.5cm and 1.5cm of 14] {};
              \node[main node] (16) [above right=0.5cm and 1.5cm of 14] {};
              \node[main node] (b3) [above right=0.5cm and 1.5cm of 15] {};
              \node[main node] (b4) [above right=0.5cm and 1.5cm of b3] {};
              \node[main node] (b5) [above left=0.5cm and 1.5cm of b4] {};
              \node[main node] (b6) [above left=0.5cm and 1.5cm of b5] {};
              \node[main node] (b7) [above left=0.5cm and 1.5cm of b6] {};
              \node[main node] (b8) [above left=0.5cm and 1.5cm of b7] {};
              \node[main node] (b9) [above right=0.5cm and 1.5cm of b7] {};
              \node[main node] (b10) [above right=0.5cm and 1.5cm of b8] {};
              \node[main node] (b11) [above right=0.5cm and 1.5cm of b10] {};
              \node[main node] (b12) [above right=0.5cm and 1.5cm of b11] {};
              \node[main node] (b13) [above right=0.5cm and 1.5cm of b12] {};
              \node[main node] (b14) [above left=0.5cm and 1.5cm of b13] {};
              \node[main node] (b15) [above left=0.5cm and 1.5cm of b14] {};
              \node[main node] (b16) [above right=0.5cm and 1.5cm of b14] {};
              \node (b17) [above right=0.5cm and 1.5cm of b15] {};
              \node [above=0.01cm of 8] {$a$}; 
              \node [above=0.01cm of 9] {$b$}; 
              \node [above=0.01cm of 10] {$v_0$}; 
              \node [above=0.01cm of 11] {$v_{1}$}; 
              \node [above=0.01cm of 12] {$v_{n-1}$}; 
              \node [above=0.01cm of 13] {$v_n$}; 
              \node [above=0.01cm of 14] {$v_{n-1}$}; 
              \node [above=0.01cm of 15] {$v_{n-2}$}; 
              \node [above=0.01cm of 16] {$v_{n}$}; 
              \node [above=0.01cm of b3] {$v_{n-1}$}; 
              \node [above=0.01cm of b4] {$v_{n}$}; 
              \node [above=0.01cm of b5] {$v_{n-1}$}; 
              \node [above=0.01cm of b6] {$v_1$}; 
              \node [above=0.01cm of b7] {$v_0$}; 
              \node [above=0.01cm of b8] {$a$}; 
              \node [above=0.01cm of b9] {$b$}; 
              \node [above=0.01cm of b10] {$v_0$}; 
              \node [above=0.01cm of b11] {$v_{1}$}; 
              \node [above=0.01cm of b12] {$v_{n-1}$}; 
              \node [above=0.01cm of b13] {$v_n$}; 
              \node [above=0.01cm of b14] {$v_{n-1}$}; 
              \node [above=0.01cm of b15] {$v_{n-2}$}; 
              \node [above=0.01cm of b16] {$v_{n}$}; 
              \node [above=1.2cm of b15] {$\vdots$};

              \path[draw]
             
              (8)--(10)--(11)
              (9)--(10)
              (12)--(13)--(14)--(15)--(b3)
              (14)--(16)--(b3)

              (b3)--(b4)--(b5)
              (b6)--(b7)--(b8)--(b10)
              (b7)--(b9)--(b10)--(b11)
              
              (b12)--(b13)--(b14)
              (b15)--(b14)--(b16);

              \path[draw,dashed]
              (11)--(12)
              (b11)--(b12)
              (b5)--(b6)
              (b17)--(b16)
              (b17)--(b15);

              \node (r1) [below left = 0.2cm and 1.2cm of 8] {};
              \node (r2) [below right = 0.2cm and 4.5cm of 9] {};
              \node (r3) [above = 3.5cm of r1] {};
              \node (r4) [above = 3.5cm of r2] {};
              \node (r5) [above = 3.3cm of r3] {};
              \node (r6) [above = 3.3cm of r4] {};
              \node (r7) [above = 3.25cm of r5] {};
              \node (r8) [above = 3.25cm of r6] {};
              \node (r9) [above = 2.4cm of r7] {};
              \node (r10) [above = 2.4cm of r8] {};

              \path[draw,dashed,blue]
              (r1)--(r2)--(r4)--(r3)--(r1)
              (r4)--(r6)--(r5)--(r3)
              (r5)--(r7)--(r9)
              (r6)--(r8)
              (r7)--(r8)--(r10);
            \end{tikzpicture}
          }
        \end{center}
  
        \hspace{12pt} We shall prove that
        \begin{equation}
          w_kab\le_R w_{k+2}a\le_R w_{k+2}\le_R w_{k+1}\le_R w_kab 
          \label{eq:ab1}
        \end{equation}
        for all even integers $k\ge 0$ and that
        \begin{equation}
          w_kv_{n-2}v_n\le_R w_{k+2}v_{n}\le_R w_{k+2}\le_R w_{k+1}\le_R
          w_kv_{n-2}v_n
          \label{eq:n-2n}
        \end{equation}
        for all odd integers $k\ge 1$.  It then follows that $w_k\sim_R ab$ and
        hence $\la(w_k)=\la(ab)=2$ for all $k\ge 1$, therefore $W$ is
        $\la(2)$-infinite.

        We first prove \eqref{eq:ab1}. As in the previous lemma, we
        may reduce this to
        showing $\mu(w_kab,w_{k+2}a)\neq 0$. Applying Proposition \ref{star and
        mu} with respect to the pairs $\{a,v_0\}, \{v_0,v_1\},\cdots,
        \{v_{n-2},v_{n-1}\}$ successively, we get
        $\mu(w_kab,w_{k+2}a)=\mu(x,y)$ where
        \[
          x=w_kabv_0v_1\cdots v_{n-2}v_{n-1},
          y=w_{k+1}v_{n-2}v_nv_{n-1}v_nv_{n-1}.
        \]
        Furthermore, like before, but this time using the coset decomposition of
        $x$ and $y$ with respect to $I=\{v_{n-1},v_n\}$, we get 
        \[
          \mu(x,y)=\mu(w_kabv_0v_1\cdots
          v_{n-2}v_{n-1}v_nv_{n-1}v_{n},w_{k+1}v_{n-2}v_n)=\mu(w_{k+1}v_n,w_{k+1}v_{n-2}v_n).
        \]
        Corollary \ref{difference 1} then implies that $\mu(x,y)=1\neq 0$,
        which proves \eqref{eq:ab1}.

        We now prove \eqref{eq:n-2n}. We may reduce this to showing that
        $\mu(w_kv_{n-2}v_n,w_{k+2}v_{n})\neq 0$. By setting
        $x=w_kv_{n-2}v_n, y=w_{k+2}v_n$, considering their coset decompositions with
        respect to $I=\{v_{n-1},v_n\}$, and defining $[i,j]$ for $0\le i,j\le
        4$ in the usual way, we again have
        \[
          \mu(x,y)=[1,4]=
          [4,1]=\mu(w_kv_{n-2}v_nv_{n-1}v_nv_{n-1},w_{k+1}abv_0v_1\cdots
          v_{n-2}v_{n-1}).
        \]
        Finally, applying Proposition \ref{star and mu} repeatedly with respect
        to the suitable pairs of vertices, we have
        \[
       \mu(w_kv_{n-2}v_nv_{n-1}v_nv_{n-1},w_{k+1}abv_0v_1\cdots
        v_{n-2}v_{n-1})=\mu(w_{k+1}a,w_{k+1}ab)=1
        \]
where the last equality follows from Corollary \ref{difference 1}. This
implies that $\mu(x,y)\neq 0$, and our proof is now complete.
       \end{proof}


      \subsection{Finishing the proof}
      We may now combine the lemmas to finish the proof of the ``only if''
      directions of Theorem \ref{first
      theorem}.(2). We first deal with the case where $G$ contains a cycle:
      \begin{prop}
        \label{necessary cycle}
        Let $W$ be a Coxeter group with Coxeter diagram $G$.
        If $G$ contains a cycle and $W$ is $\la(2)$-finite, then $W$ is a
        complete graph.
      \end{prop}
      

      \begin{proof}
Consider a cycle $C=(v_1,v_2,\cdots, v_n,v_1)$ of maximal length in $G$. We
        claim that $v_1,v_2,\cdots,v_n$ must be all the vertices of $G$. To see
        this, suppose otherwise and let $v$ be any other vertex of $G$ not in $C$.
        Then $v$ must be adjacent to all vertices in $C$ by Part (1) of Lemma
        \ref{cycle witness}. However, in this case
        $C'=(v,v_1,v_2,\cdots,v_n,v)$ would
        form a longer cycle in $G$ than $C$, contradicting our maximality assumption. 

        To prove $G$ is complete, it now suffices to show that $v_i$ and $v_j$ are
        adjacent for all $1\le i,j\le n$. This follows from Part (2) of Lemma
        \ref{cycle witness}, which says that otherwise $W$ would be
        $\la(2)$-infinite.
      \end{proof}

      \begin{remark}
      Note that the above proposition is
      slightly stronger than the ``only if'' condition of Theorem
      \ref{first theorem}.(1) since
 we do not need to assume $W$ is irreducible in its
 statement or proof. This is because Lemma
 \ref{cycle witness} implies that the diagram of any $\la(2)$-finite
        Coxeter group must be connected if it contains a cycle. 
      \end{remark}

      Next, we deal with Part (2) of the theorem. For
      convenience, we define a \emph{path graph} to be a weighted graph such that
      the underlying unweighted graph looks like a ``straight line'', i.e. a graph
      of type $A_n$ from Figure \ref{fig:finite E}.

      \begin{prop}
        \label{necessary prop}
        Let $W$ be an irreducible Coxeter group with Coxeter diagram $G$. If
        $G$ is acyclic and $W$ is $\la(2)$-finite, then
        $G$ is one of the graphs from Figure \ref{fig:finite E}.
      \end{prop}

      \begin{proof}
        Suppose $G$ is acyclic and $W$ is $\la(2)$-finite. Since $W$ is
        irreducible, $G$ is connected and hence a tree. 

        Let $h$ be the largest weight of an edge in $G$. This is well-defined
        because $G$ contains finitely many vertices and hence edges. If $h\ge
        6$, all other edges in $G$ must have weight 3, for otherwise we must be
        able to find a
        subgraph of the form shown in Lemma \ref{two strong bonds} so that $W$
        would be $\la(2)$-infinite. Lemma
        \ref{strength 6} then further implies that $G$ must be exactly of rank 2,
        therefore $G$ is of the form $I_2(h)$ from Figure \ref{fig:finite E}.

        Next, suppose $h=5$. Then again, in light of Lemma \ref{two strong bonds}, $G$ must have only one edge of weight 5, and all other edges of $G$ must have
        weight 3. Moreover, no vertices of $G$ can have degree 3 or higher by
        lemmas \ref{small 5} and
        \ref{large 5}, therefore $G$ must be a path graph. By Lemma
        \ref{middle 5}, the unique edge of weight 5 cannot have edges both to its left
        and to its right in the path graph, therefore $G$ is 
        either $I_2(5)$ or $H_n$ for some $n\ge 3$. 

        Now suppose $h=4$. We claim that $G$ must be a path graph. Otherwise,
        let $v$ be a vertex of degree at least $3$, then either $v$ is incident to at
        least two
        edges of weight 4 and $W$ is $\la(2)$-infinite by Lemma \ref{affine
        C3}, or, if $v$ is incident to one or no edge of weight 4, 
        then $W$ is $\la(2)$-infinite by Lemma \ref{affine B} as $G$ must contain a
        subgraph of the form shown in the lemma. Given that $G$ is a path
        graph, we also claim that $G$ can contain at
        most two edges of weight 4. Otherwise, since $G$ cannot contain a
        subgraph of
        the form shown in Lemma \ref{affine C3}, there must be at least one of
        weight 3 between each pair of
        edges of weight 4. This would force $G$ to contain a subgraph of the form
        shown in Lemma \ref{nonextreme 4s}, so $W$ would be
        $\la(2)$-infinite, a
        contradiction. Moreover, lemmas  \ref{affine C3},
        \ref{affine C4} and \ref{nonextreme 4s} also imply that in the
        case where $G$ contains two edges of weight 4, they must appear on the two
        ends of the path graph and $G$ must of the form $\tilde C_n$ for some
        $n\ge 5$. Finally, if $G$ contains exactly one edge of weight 4, then Lemma
        \ref{affine F5 lemma} implies that $G$ must be of the form $B_n$ for some
        $n\ge 2$ or $F_n$ for some $n\ge 4$.

        Finally, we consider the case $h=3$. If $G$ contains no vertex of degree 3
        or higher, then $G$ is a path graph and hence of the form $A_n$ from Figure
        \ref{fig:finite E} for some $n\ge 2$. On the other hand, by Lemma \ref{two
        branches}, $G$ cannot contain any vertex of degree 4, nor can it contain two
        vertices of degree at least 3, therefore if $G$ has a vertex of degree at
        least 3 at all, $G$ must be of the form shown in Figure \ref{mustbe
        graph}, where removal of the trivalent vertex results in three path graphs containing
        $p,q,r$ vertices for some $1\le p\le q\le r$. Note that $p\ge 2$ would
        imply that $G$ contains a subgraph of the form shown in Lemma \ref{affine E6}, therefore
        $p=1$ by the lemma. But then $G$ is of the form $E_{q,r}$ from Figure
        \ref{fig:finite E}.  This completes our proof. 
      \end{proof}

      \begin{figure}
        \begin{tikzpicture}

            \node[main node] (1) {};
            \node[main node] (0) [left=0.75cm of 1] {};
            \node[main node] (00) [left=0.5cm of 0] {};
            \node[main node] (2) [right=0.5cm of 1] {};
            \node[main node] (3) [right=0.5cm of 2] {};
            \node[main node] (6) [right=0.75cm of 3] {};
            \node[main node] (66)[right=0.5cm of 6] {};
            \node[main node] (8) [above=0.5cm of 2] {};
            \node[main node] (9) [above=0.75cm of 8] {};
            \node[main node] (99) [above=0.5cm of 9] {};

            \path[draw]
            (00)--(0)
            (1)--(2)--(3)
            (6)--(66)
            (2)--(8)
            (9)--(99);
            \path[draw,dashed]
            (8)--(9)
            (3)--(6)
            (0)--(1);
          \end{tikzpicture}
          \caption{}
          \label{mustbe graph}
        \end{figure}

             By proving propositions \ref{sufficient cycle}, \ref{sufficient prop}, \ref{necessary cycle} and
      \ref{necessary prop}, we have now completed the proof of Theorem
      \ref{first theorem}.
 
      \section{Reducible $\la(2)$-finite Coxeter groups}
      \label{reducible}
      We now prove Theorem \ref{second theorem}, which is restated below for
      convenience.

      \begin{thm*}
  Let $W$ be a reducible Coxeter group with Coxeter diagram $G$. Let $G_1,G_2,\cdots, G_n$ be the
  connected components of $G$, and let $W_1,W_2,\cdots, W_n$ be their
  corresponding Coxeter groups, respectively.  Then the following are equivalent.
  \begin{enumerate}
    \item $W$ is $\la(2)$-finite.
    \item The number $n$ is finite, i.e. $G$ has finitely many connected components, and
    $W_i$ is both $\la(1)$-finite and
      $\la(2)$-finite for each $1\le i\le n$. 
    \item The number $n$ is finite, and for each $1\le i\le n$, $G_i$ is a graph of the form $A_n (n\ge 1), B_n
      (n\ge 2),
      E_{q,r} (q,r\ge 1), F_n (n\ge 4), H_n(n\ge 3)$ or $I_2(m) (5\le m\le \infty)$, i.e.
      $G_i$ is a graph from Figure \ref{fig:finite E} other than $\tilde{C}_n
      (n\ge 5)$.
  \end{enumerate}
\end{thm*}
      \begin{proof}
        The equivalence of (2) and (3) is immediate from Proposition \ref{tree} and
        Theorem \ref{first theorem}, so we just need to prove the equivalence of (1)
        and (2). 
        
        We first prove that (1) implies (2).  Suppose $W$ is $\la(2)$-finite. Then clearly $W_i$ is $\la(2)$-finite for each $i$.
     Also note that if we pick an element $t_i$ from the generating set of $W_i$ for each
     $i$, then $\la(t_it_j)=2$ for all
     distinct $1\le i,j\le n$ by Corollary \ref{product a}, therefore $n$ must
     be finite. Finally, let $1\le i\le n$,
     and consider words of the form $tw_1$ where $w_1=s_1s_2\cdots s_q$ is the reduced word
        of an element of $\la$-value 1 in $W_i$ and $t$ is a vertex in
        $G_j$ for some $j\neq i$.  Clearly, the word $tw_1$ is still
        reduced. Furthermore, since
        $w_1$ must have a unique reduced word by Proposition \ref{subregular}, no two adjacent
        letters in $w_1$ can commute, i.e. $m(s_k,s_{k+1})\ge 3$, for all $1\le k\le
        q-1$. This means $tw_1$ can be reduced to $ts_1$ via suitable lower star operations  by Remark
        \ref{removal}, therefore $\la(tw_1)=\la(ts_1)=2$ by corollaries
        \ref{star a} and \ref{product a}.
        It follows that $W_i$ is $\la(1)$-finite, for otherwise we
        can find infinitely many distinct elements of the form $tw_1$ in $W$.

        It remains to prove that (2) implies (1). Suppose $W_i$ is both
        $\la(1)$-finite and $\la(2)$-finite for each $1\le i\le n$, and let
        $w\in W$ be an
        element of $\la$-value 2. Since  every generator of $W_i$ commutes with every generator
        of $W_j$ for any distinct $i,j$, $w$ admits a reduced word $w=w_1\cdot
        w_2\cdot\cdots\cdot w_n$ where each $w_i$ is a (possibly empty) reduced
        word for an element in $W_i$. Note that at most two of of $w_1,\cdots, w_n$
        can be nonempty, for otherwise if we have reduced words $w_i=r_1r_2\cdots, w_j=s_1s_2\cdots, w_k=t_1t_2\cdots$ for
        some $i<j<k$, then we may commute $s_1$ and $t_1$ past letters to their left
        to form the reduced word of the form $ w=w_1\cdots
        w_{i-1}(r_1s_1t_1)r_2\cdots$, therefore $\la(w)\ge \la(r_1s_1t_1)=3$ by
        corollaries \ref{weak order} and \ref{product a}. It follows $w$ must be
        of the form $w=w_i\cdot w_j$ for some $1\le i<j\le n$. By Corollary
        \ref{weak order}, this forces
        $\la(w_i)\le 2$ and $\la(w_j)\le 2$ now that $\la(w)=2$. Since $W_i$ is
        both $\la(1)$-finite and $\la(2)$-finite for
        all $1\le i\le n$, this implies that there are
        only finitely many possibilities for $w$, therefore $W$ is
        $\la(2)$-finite. This
        completes the proof.  
      \end{proof}

      \section{Concluding remarks}
      \label{conclusion}

We discuss below some open problems naturally arising from this paper. 

\subsection{Generalizing Shi's result}
The ``$\la=n$'' result of Shi from Proposition \ref{a from heap}, that is, the
result
that $\la(w)=n(w)$ for every fully commutative element $w$ in a Weyl
or affine Weyl group, is a very powerful tool used in this paper. One can
verify that all
witnesses of $\la$-value 2 given in Section \ref{sec:necessity} have
$n$-value 2, therefore we actually have $\la(w)=n(w)$ for every
witness $w$ we used, regardless of
whether we showed $\la(w)=2$ by using the result $\la(w)=n(w)$. In fact, 
the $\la=n$ result is exactly how we found the witnesses heuristically:
we constructed heaps of fully commutative elements with $n$-value 2, then tried
to verify that the elements have $\la$-value 2 by using various methods.

We extended Shi's
result to star reducible Coxeter groups in Proposition \ref{star reducible a},
and we believe it would be very interesting to know whether the result can be
further generalized to arbitrary Coxeter groups. As direct computations of
$\la$-values require understanding products of Kazhdan--Lusztig basis elements in Hecke
algebras, which is usually difficult, the generalization of the $\la=n$ result would provide a
powerful shortcut to computing $\la$-values for fully commutative elements.
Moreover, in Section 6.6 of the paper \cite{BJN}, Biagioli, Jouhet
and Nadeau mention that it would be
interesting to explore statistics on fully commutative elements which can be
studied naturally on heaps; the generalization would suggest that the $n$-value is a
  such a statistic with interesting connections to representation theory.

On a more technical level, a generalization of the $\la=n$ result would also be remarkable in the
following sense. 
If a word $w$ represents a fully commutative element in a Coxeter
group $W$ with Coxeter diagram $G$, then it also represents a fully commutative
element in any Coxeter group $W'$ whose Coxeter diagram $G'$ is obtained from
$G$ by increasing the weights of some edges. The increase in edge
weights does not affect the heap of $w$, hence these two elements have the
same $n$-value, yet it is not obvious why the increase should not affect the
$\la$-value.

\subsection{Enumerating elements of $\la$-value 2} Given our classification of
$\la(2)$-finite Coxeter groups, it is natural to wonder 
how many elements of
  $\la$-value 2 there are in each group $W$ from the classification. This can
  be related to the following questions. First, how is the set of the elements of $\la$-value 2 in $W$ partitioned into
 two-sided Kazhdan--Lusztig cells? 
 Second, for each two-sided Kazhdan--Lusztig cell $E$ of $W$ with $\la$-value
 2, how can we describe the structure of the associated subalgebra $J_E$ of the asymptotic
 Hecke algebra $J$ of $W$? By definition, the algebra $J_E$ is a free abelian group with a basis
 indexed by the elements of $E$, therefore such a 
description would allow us to recover the cardinality of $E$ as the rank of
$J_E$ and understanding all subalgebras of the form $J_E$ where $E$ is a cell
of $\la$-value 2 would allow us to recover the total number of elements of
$\la$-value 2. 

We intend to investigate the above questions
elsewhere.
We should mention that
the analogous questions
for elements of $\la$-value 1 have been solved in the following sense. Let
$W$ be an arbitrary Coxeter group, let 
$C=\{w\in W:\la(w)=1\}$, and let $U=\{w\in W: w\;\text{has a unique reduced
word}\}$. Then we have $U=C\sqcup \{1_W\}$ (see Proposition
\ref{subregular}), therefore $W$ is $\la(1)$-finite if and
only if $U$ is finite, and the problem of counting $C$ for
$\la(1)$-finite Coxeter 
groups is equivalent to that of counting $U$ for those Coxeter groups where $U$ is finite. The latter problem 
is solved by Hart in \cite{Hart}. As for the partition of $C$ into cells, it is
well-known that $C$ always forms a single two-sided cell. In addition, the
structure of the algebra $J_C$ for general Coxeter groups is studied in \cite{Xu}, and the structure of
$J_C$ for $\la(1)$-finite Coxeter groups can be easily deduced from the results of
\cite{Xu}. 

\subsection{Classification of $\la(3)$-finite Coxeter groups}
Yet another natural way to extend the results of this paper is to classify
$\la(3)$-finite Coxeter groups. Note that unlike
elements of $\la$-value 2, elements of $\la$-value 3 in a Coxeter group
are not necessarily fully
commutative (cf. Proposition \ref{a2 is fc}); for example, in the Coxeter group
of type $A_2$, the longest element has $\la$-value
3 by Proposition \ref{a for longest} but is not fully commutative.
This means that in general an element of $\la$-value 3 does not have an intrinsically
associated heap. As heaps of fully commutative elements are fundamental to our
arguments in this paper, a classification of $\la(3)$-finite
Coxeter
groups will likely require a different set of ideas from those used in this paper.

      \longthanks{
We would like to thank Victor Ostrik for stimulating interest in the
        subject of this paper. We thank Jian-Yi Shi for
        helpful discussions. The first named author would like to thank
        Ben Elias for inviting him to the Department of Mathematics at the
        University of Oregon, where the paper was initiated, and to thank the
        department for hospitality. }
        \bibliographystyle{amsplain-ac}
 \bibliography{a_ref.bib}

\providecommand{\bysame}{\leavevmode\hbox to3em{\hrulefill}\thinspace}
\providecommand{\MR}{\relax\ifhmode\unskip\space\fi MR }
\providecommand{\MRhref}[2]{%
  \href{http://www.ams.org/mathscinet-getitem?mr=#1}{#2}
}
\providecommand{\href}[2]{#2}
\begin{thebibliography}{10}

\bibitem{Bez1}
R.~Bezrukavnikov, \emph{On tensor categories attached to cells in affine {W}eyl
  groups}, in Representation theory of algebraic groups and quantum groups,
  Adv. Stud. Pure Math., vol.~40, Math. Soc. Japan, Tokyo, 2004, pp.~69--90.

\bibitem{Bez3}
R.~Bezrukavnikov, M.~Finkelberg, and V.~Ostrik, \emph{On tensor categories
  attached to cells in affine {W}eyl groups. {III}}, Israel J. Math.
  \textbf{170} (2009), 207--234.

\bibitem{Bez2}
R.~Bezrukavnikov and V.~Ostrik, \emph{On tensor categories attached to cells in
  affine {W}eyl groups. {II}}, in Representation theory of algebraic groups and
  quantum groups, Adv. Stud. Pure Math., vol.~40, Math. Soc. Japan, Tokyo,
  2004, pp.~101--119.

\bibitem{BJN}
R.~Biagioli, F.~Jouhet, and P.~Nadeau, \emph{Fully commutative elements in
  finite and affine {C}oxeter groups}, Monatsh. Math. \textbf{178} (2015),
  no.~1, 1--37.

\bibitem{BJ}
S.~C. Billey and B.~C. Jones, \emph{Embedded factor patterns for {D}eodhar
  elements in {K}azhdan--{L}usztig theory}, Ann. Comb. \textbf{11} (2007),
  no.~3-4, 285--333.

\bibitem{BB}
A.~Bj{\"o}rner and F.~Brenti, \emph{Combinatorics of {C}oxeter groups},
  Graduate Texts in Mathematics, vol. 231, Springer, New York, 2005.

\bibitem{EW}
B.~Elias and G.~Williamson, \emph{The {H}odge theory of {S}oergel bimodules},
  Ann. of Math. (2) \textbf{180} (2014), no.~3, 1089--1136.

\bibitem{Ernst}
D.~C. Ernst, \emph{Diagram calculus for a type affine {$C$} {T}emperley-{L}ieb
  algebra, {II}}, J. Pure Appl. Algebra \textbf{222} (2018), no.~12,
  3795--3830.

\bibitem{EGNO}
P.~Etingof, S.~Gelaki, D.~Nikshych, and V.~Ostrik, \emph{Tensor categories},
  Mathematical Surveys and Monographs, vol. 205, American Mathematical Society,
  Providence, RI, 2015.

\bibitem{Fan}
C.~K. Fan, \emph{A {H}ecke algebra quotient and some combinatorial
  applications}, J. Algebraic Combin. \textbf{5} (1996), no.~3, 175--189.

\bibitem{Geck_cellular}
M.~Geck, \emph{Hecke algebras of finite type are cellular}, Invent. Math.
  \textbf{169} (2007), no.~3, 501--517.

\bibitem{Geck_a}
M.~Geck and L.~Iancu, \emph{Lusztig's {$a$}-function in type {$B_n$} in the
  asymptotic case}, Nagoya Math. J. \textbf{182} (2006), 199--240.

\bibitem{Green_interval}
R.~M. Green, \emph{On rank functions for heaps}, J. Combin. Theory Ser. A
  \textbf{102} (2003), no.~2, 411--424.

\bibitem{star_reducible}
\bysame, \emph{Star reducible {C}oxeter groups}, Glasg. Math. J. \textbf{48}
  (2006), no.~3, 583--609.

\bibitem{Green_Jones}
\bysame, \emph{Generalized {J}ones traces and {K}azhdan--{L}usztig bases}, J.
  Pure Appl. Algebra \textbf{211} (2007), no.~3, 744--772.

\bibitem{FCcells}
R.~M. Green and J.~Losonczy, \emph{Fully commutative {K}azhdan--{L}usztig
  cells}, Ann. Inst. Fourier (Grenoble) \textbf{51} (2001), no.~4, 1025--1045.

\bibitem{Hart}
S.~Hart, \emph{How many elements of a {C}oxeter group have a unique reduced
  expression?}, J. Group Theory \textbf{20} (2017), no.~5, 903--910.

\bibitem{KL}
D.~Kazhdan and G.~Lusztig, \emph{Representations of {C}oxeter groups and
  {H}ecke algebras}, Invent. Math. \textbf{53} (1979), no.~2, 165--184.

\bibitem{subregular}
G.~Lusztig, \emph{Some examples of square integrable representations of
  semisimple {$p$}-adic groups}, Trans. Amer. Math. Soc. \textbf{277} (1983),
  no.~2, 623--653.

\bibitem{Mu}
\bysame, \emph{Cells in affine {W}eyl groups}, in Algebraic groups and related
  topics ({K}yoto/{N}agoya, 1983), Adv. Stud. Pure Math., vol.~6,
  North-Holland, Amsterdam, 1985, pp.~255--287.

\bibitem{Lusztig_II}
\bysame, \emph{Cells in affine {W}eyl groups. {II}}, J. Algebra \textbf{109}
  (1987), no.~2, 536--548.

\bibitem{LG}
\bysame, \emph{Hecke algebras with unequal parameters},
  \url{arXiv:math/0208154}, 2014.

\bibitem{Shi}
J.~Shi, \emph{Fully commutative elements in the {W}eyl and affine {W}eyl
  groups}, J. Algebra \textbf{284} (2005), no.~1, 13--36.

\bibitem{FC}
J.~R. Stembridge, \emph{On the fully commutative elements of {C}oxeter groups},
  J. Algebraic Combin. \textbf{5} (1996), no.~4, 353--385.

\bibitem{FC2}
J.~R. Stembridge, \emph{The enumeration of fully commutative elements of
  {C}oxeter groups}, J. Algebraic Combin. \textbf{7} (1998), no.~3, 291--320.

\bibitem{Tits}
J.~Tits, \emph{Le probl\`eme des mots dans les groupes de {C}oxeter}, in
  Symposia {M}athematica ({INDAM}, {R}ome, 1967/68), {V}ol. 1, Academic Press,
  London, 1969, pp.~175--185.

\bibitem{Warrington}
G.~S. Warrington, \emph{Equivalence classes for the {$\mu$}-coefficient of
  {K}azhdan--{L}usztig polynomials in {$S_n$}}, Exp. Math. \textbf{20} (2011),
  no.~4, 457--466.

\bibitem{Xu}
T.~Xu, \emph{On the subregular {J}-rings of {C}oxeter systems}, Algebras and
  Representation Theory (2018).

\end{thebibliography}

      \end{document}